\documentclass[12pt,reqno]{amsart}

\usepackage{amsmath,amssymb,amsthm}
\usepackage{microtype}
\usepackage[hidelinks]{hyperref}
\usepackage{enumitem}
\usepackage{graphicx}

\newtheorem{theorem}{Theorem}[section]
\newtheorem{lemma}[theorem]{Lemma}
\newtheorem{proposition}[theorem]{Proposition}
\newtheorem{corollary}[theorem]{Corollary}
\theoremstyle{remark}

\numberwithin{equation}{section}

\newcommand{\T}{\mathbb T}
\newcommand{\R}{\mathbb R}
\newcommand{\A}{\mathcal A}
\newcommand{\M}{\mathcal M}
\newcommand{\dd}{\mathrm d}

\title[Sharp Convergence Rates for the Vanishing Discount]
{Sharp Convergence Rates for the Vanishing Discount Problem with Hyperbolic Aubry Sets}

\author{Panrui Ni}
\address[P. Ni]{
Department of Mathematics, Faculty of Science and Engineering,
Waseda University, 3-4-1 Okubo, Shinjuku-ku, Tokyo 169-8555, Japan;
Shanghai Center for Mathematical Sciences, Fudan University,
Shanghai 200438, China}
\email{panruini@gmail.com}

\date{\today}


\begin{document}

\begin{abstract}
Let $H\in C^2(T^*M)$ be a Tonelli Hamiltonian on a closed connected manifold and let
$u_\lambda$ solve
\[
  \lambda u_\lambda+H(x,Du_\lambda)=c(H)\qquad\text{in }M.
\]
We study the convergence rate of $u_\lambda$ to the selected critical
solution $u_0$.  Assume that the lifted Aubry set is a finite union $\widetilde{\A}=\Gamma_1\sqcup\cdots\sqcup\Gamma_N$, where each $\Gamma_i$ is either a hyperbolic equilibrium or a periodic orbit
hyperbolic in the critical energy level.  We prove
\[
  -C\lambda\le u_\lambda-u_0\le C\lambda|\log\lambda|.
\]
Let $\mu_i$ be the projected Mather measure associated with
$\Gamma_i$. If
\[
  \int_M u_0\,\dd\mu_i=0\qquad \text{for every }i,
\]
 then
\[
  \|u_\lambda-u_0\|_\infty\le C\lambda.
\]
In particular, if the lifted Aubry set consists of a single hyperbolic equilibrium or a single hyperbolic periodic orbit, the convergence rate is $O(\lambda)$.
We give examples showing that both convergence rates $O(\lambda)$ and $O(\lambda|\log\lambda|)$ are optimal. Without hyperbolicity, finite-order
degenerate examples give lower bounds of order
$\lambda^{1/(2r-1)}$ with $r\ge2$.  We also construct infinite-order
degenerate examples with
arbitrarily slow convergence. Taken together, these results provide, to our knowledge, the first systematic quantitative theory for the vanishing discount problem in the Tonelli setting.
\end{abstract}

\maketitle

\section{Introduction}

The vanishing discount approximation for Hamilton--Jacobi equations goes back
to the ergodic approximation introduced in the study of the cell problem;
see \cite{LPV}.  Consider
\begin{equation*}\tag{E$_\lambda$}\label{el}
  \lambda u_\lambda+H(x,Du_\lambda)=c(H)\qquad\text{in }M
\end{equation*}
and
\begin{equation*}\tag{E$_0$}\label{e0}
  H(x,Du)=c(H)\qquad\text{in }M.
\end{equation*}
The first issue is qualitative: does the whole family $u_\lambda$ converge as $\lambda\to 0^+$, and if so,
which solution of \eqref{e0} is selected?

Early selection results were obtained in \cite{GomesSelection} and,
under a hyperbolicity assumption, in 
\cite{IturriagaSM}. The latter proved convergence when the Aubry set consists
of finitely many hyperbolic critical points.  A general solution of the qualitative problem in the convex setting was
subsequently given in \cite{DFIZ} and \cite{MT}.
For continuous, coercive and convex Hamiltonians on a compact manifold, it was proved in \cite{DFIZ} that
$u_\lambda$ converges along the whole family $\lambda \to0^+$, and the selected
limit is characterized in terms of Peierls barriers and Mather measures.
It is worth mentioning that \cite{MT} uses the nonlinear adjoint method to treat the degenerate viscous term and includes the first-order
case. An alternative
dual viewpoint based on viscosity Mather measures was developed in \cite{IMT}.

Since then, the qualitative selection theory has been extended in many
directions.  These include discrete models \cite{DFIZdiscrete,NZ}, boundary value problems
\cite{IMT2,TuZhangState}, vanishing contact structures \cite{CCIKZ,WYZ}, nonlinear, degenerate, or
non-monotone discount terms \cite{CFZZ,DNYZ,NiPAMS,NYZ,Zav}, and noncompact domains
\cite{IshiiSiconolfi,TZ}.
Convexity remains essential
for the convergence result, and without it, convergence may fail
\cite{Ziliotto}. 

The quantitative problem has not undergone a parallel development.  A weak KAM theory for discounted
Hamilton--Jacobi equations was developed in \cite{MitakeSoga}. As one application, error estimates were obtained
under additional dynamical assumptions, both in one dimension and under special higher-dimensional assumptions.
Related work \cite{HJZZ} studied smooth
discounted subsolutions under hyperbolicity assumptions and used them to
control the long-time convergence of discounted Lax--Oleinik semigroups.
These results suggest that the minimizing dynamics can control the error, but
they do not give a sharp rate theory for general Aubry
sets. It is also worth comparing this with the large-time behavior of the Lax--Oleinik semigroup. In the vanishing discount problem considered here, hyperbolic equilibria and hyperbolic periodic orbits lead to the same rate estimates. By contrast, a time-independent example was constructed in \cite{WangYan} with a single hyperbolic periodic orbit in the Aubry set showing a different behavior in the large-time problem. Related rate
questions for infinite-horizon discounted optimal control have been studied
under different structural assumptions; see, for example,
\cite{CGMQ,GruneWirth}.  To our knowledge, no sharp convergence-rate theory
has been available for the vanishing discount problem with Tonelli Hamiltonians.  This is in
contrast with the extensive rate theories developed for other singular
limits of Hamilton--Jacobi equations, such as vanishing viscosity and periodic
homogenization; see \cite{CDI,CirantGoffi,MTY,QSTY,TranYu} and the references
therein.

\smallskip

The purpose of this paper is to develop such a quantitative theory when the
Aubry set is a finite union of hyperbolic components. Throughout the paper $M$ is a closed connected smooth manifold and
$H\in C^2(T^*M)$ is Tonelli. For $\lambda>0$, let $u_\lambda$ be the solution of \eqref{el} and let $u_0$ be the limit of $u_\lambda$.  We write $\A\subset M$ and
$\widetilde{\A}\subset T^*M$ for the projected and lifted Aubry sets,
respectively.
\smallskip

 Assume that
\begin{itemize}
\item[(A)] $\widetilde{\A}=\Gamma_1\sqcup\cdots\sqcup\Gamma_N$, where each $\Gamma_i$ is the lift of a static class and is either a hyperbolic equilibrium
or a periodic orbit hyperbolic in $H^{-1}(c(H))$. 
\end{itemize}
Each $\Gamma_i$ carries a unique Mather measure. We denote its projection
to $M$ by $\mu_i$ and set  $m_i:=\int_M u_0\,\dd\mu_i$.
By the selection characterization in \cite[Proposition 1.3]{DFIZ}, $m_i\le0$. We call $\Gamma_i$
\emph{active} when $m_i=0$. Lemma~\ref{lem:active-representation} shows that active components represent \(u_0\) through the Peierls barrier. Thus the quantities \(m_i\) enter both the selection mechanism and the convergence rate.

\smallskip

Our main result is the following.

\begin{theorem}\label{thm:finite-main}
Assume {\rm (A)}.
Then there exist $C>0$ and $\lambda_0>0$ such that, for
$0<\lambda\le\lambda_0$,
\begin{equation}\label{eq:intro-finite-rate}
  -C\lambda
  \le u_\lambda(x)-u_0(x)
  \le C\lambda|\log\lambda|
  \qquad\text{for every }x\in M.
\end{equation}
If $m_i=0$ for $i=1,\dots,N$,
then
\begin{equation}\label{eq:intro-active-rate}
  \|u_\lambda-u_0\|_\infty\le C\lambda.
\end{equation}
\end{theorem}

\begin{corollary}\label{cor:single-intro}
Assume that $\widetilde{\A}$ consists of either a single hyperbolic
equilibrium or a single periodic orbit which is hyperbolic in its energy
level.  Then
\[
  \|u_\lambda-u_0\|_\infty\le C\lambda
\]
for all sufficiently small $\lambda>0$.
\end{corollary}

Corollary \ref{cor:single-intro} is also closely related to Ma\~n\'e's conjecture, which predicts that,
generically, the Aubry set is either a hyperbolic equilibrium or a hyperbolic periodic
orbit; see \cite{CFR} and the references therein. Thus, if Ma\~n\'e's conjecture holds true, for generic Hamiltonians, the vanishing discount convergence rate is $O(\lambda)$.

\smallskip

We briefly describe the proof of Theorem~\ref{thm:finite-main}. In \cite{MitakeSoga}, the authors derive convergence rates by estimating the time needed for minimizing curves to enter a prescribed neighborhood of the corresponding sets of $\alpha$-limit points. In general, it is hard to estimate the entrance time. Instead, our proof proceeds in three steps:
\smallskip

\noindent{\it Step 1.} Integrating by parts yields \eqref{eq:upper} and \eqref{eq:lower}, which reduce the convergence rate
problem to weighted integrals of $u_0$ along a $u_0$-calibrated curve
$\gamma_0$ and a $u_\lambda$-calibrated curve
$\gamma_\lambda$.  We denote their cotangent lifts by $Z_0$ and
$Z_\lambda$, respectively.
\smallskip

\noindent{\it Step 2.} Let $X_H$ denote the Hamiltonian vector field of $H$, and let
$\pi_M:T^*M\to M$ be the canonical projection. We construct $\Phi$ near $\Gamma_i$,
with $X_H\Phi=u_0\circ\pi_M-m_i$ on $\Gamma_i$, and obtain \eqref{eq:Phi-distance}. If some $m_i<0$, this constant $m_i$ has a weighted integral of order
$1/\lambda$ on a half-line. We only have a one-sided estimate $u_\lambda(x)-u_0(x)\ge-C\lambda$. If $m_i=0$ for every $i$, by Lemma \ref{lem:finite-distance}, the weighted integrals of $u_0$ along both $\gamma_0$ and $\gamma_\lambda$ are bounded, which implies
the linear rate.
\smallskip

\noindent{\it Step 3.} For the upper bound of $u_\lambda-u_0$ we use a finite-time approximation of the
Peierls barrier; see Lemma \ref{lem:peierls-compression}. The key idea is a curve cutting argument, where hyperbolicity plays an important role; see Figure \ref{fig:peierls-compression}. Related curve cutting arguments have also proved useful in
quantitative homogenization; see \cite{TranYu}. The resulting curve has
duration $O(R)$ and an $O(e^{-\kappa R})$ action error.  Choosing $R\asymp|\log\lambda|$ yields the logarithmic rate.

\smallskip
The proof of Lemma \ref{lem:finite-distance} includes the following ingredients:
\smallskip

\noindent (i) By using the subsolution of \eqref{e0} which is strict outside the Aubry set, we define a function $q$ in \eqref{eq:q-def}. By Lemma \ref{lem:q}, $q$ is bounded from below by a positive constant away from a neighborhood $U$ of $\widetilde{\A}$
on the compact region visited by $Z_0$ and
$Z_\lambda$. By the definition of $q$, we have \eqref{eq:q0} and \eqref{eq:qlambda}. Thus, the integral of $q$ controls the total time spent outside $U$ for $Z_0$ and the corresponding weighted time for $Z_\lambda$.
\smallskip

\noindent (ii) If $Z_0$ and
$Z_\lambda$ visit the neighborhood $U$ once, the hyperbolicity estimates in
Lemmas~\ref{lem:eq-local} and~\ref{lem:per-local} give a uniform bound
for the integral of the distance for $Z_0$ and for the weighted
integral of the distance for $Z_\lambda$, even when the duration of the
visit is infinite.
\smallskip

\noindent (iii) It remains to control the visits of $Z_0$ and
$Z_\lambda$ to a neighborhood of the Aubry set.  We use two nested neighborhoods $U_\rho^i\subset U_{2\rho}^i$ of $\Gamma_i$.  Consider a connected visit to $U_{2\rho}^i$ that reaches $U_\rho^i$ and ends before time $0$.  After its last exit from $U_\rho^i$, the orbit must travel from $\partial U_\rho^i$ to $\partial U_{2\rho}^i$.  Since the speeds of $Z_0$ and $Z_\lambda$ are uniformly bounded, this crossing takes a uniformly positive amount of time.  Moreover, $q$ is bounded from below along the crossing.  Then \eqref{eq:q0} and \eqref{eq:qlambda} control the number of such visits for $Z_0$ and the weighted number of visits for $Z_\lambda$.  There is at most one additional terminal visit ending at time $0$.
\smallskip

Since there are only finitely many $\Gamma_i$, we choose the 
neighborhoods of $\Gamma_i$ pairwise disjoint and the local estimates
hold with common constants for all components $\Gamma_i$. The case of more
general hyperbolic Aubry sets 
is not treated here.

\smallskip

The two rates in Theorem~\ref{thm:finite-main} are both sharp.  An example given in Section~\ref{sec:one-well} with one hyperbolic equilibrium satisfies
\[
  \|u_\lambda-u_0\|_\infty\asymp\lambda.
\]
Another example has two hyperbolic equilibria, with
$m_1=0$ and $m_2<0$, and satisfies
\[
  \|u_\lambda-u_0\|_\infty
  \asymp\lambda|\log\lambda|.
\]
Thus, hyperbolicity by itself does not force the linear rate.  The numbers
$m_i=\int_Mu_0\,\dd\mu_i$ also have a direct role in the rate estimate.
If $m_i=0$ for every $\Gamma_i$, the weighted integrals in \eqref{eq:upper} and \eqref{eq:lower} are uniformly bounded, and the rate is linear.  When some $m_i<0$,
the
two-well example shows that the logarithmic loss can then occur.  
On the other hand, if the Aubry set consists of a single degenerate
equilibrium of order $2r$ with $r\ge2$, we construct examples with
\[
  \|u_\lambda-u_0\|_\infty
  \ge C_r\lambda^{1/(2r-1)}.
\]
Since $1/(2r-1)\to0$, no positive algebraic exponent can hold uniformly once
the hyperbolicity assumption is removed.  We also construct examples with
arbitrarily slow convergence.  More precisely, for every prescribed
nondecreasing function $\omega(\lambda)\to0$, there exists a smooth
one-dimensional example with $\A=\{0\}$ and with all derivatives of
the potential vanishing at $0$ such that
\[
  \|u_\lambda-u_0\|_\infty\ge\omega(\lambda)
\]
for all sufficiently small $\lambda>0$.

\subsection*{Organization of the paper} Section~\ref{sec:prelim} gives the
variational inequalities and the estimate obtained from a strict critical
subsolution.  Section~\ref{sec:hyperbolic} proves the local estimates near a
hyperbolic equilibrium and a hyperbolic periodic orbit.  Section~\ref{sec:finite-proof}
proves Theorem~\ref{thm:finite-main}.  Section~\ref{sec:one-well} gives the
one-well examples, showing sharpness of the rate $O(\lambda)$, finite-order degenerate examples with lower bounds of
order $\lambda^{1/(2r-1)}$, and arbitrarily slow convergence.  Section~\ref{sec:two-well}
proves sharpness of the logarithmic rate.

\section{Preliminary estimates}\label{sec:prelim}

\subsection{Variational inequalities}\label{subsec:variational}

Let $M$ be a compact connected smooth $n$-dimensional manifold without
boundary.  Fix Riemannian metrics on $M$ and on $T^*M$, and denote the
corresponding distances by $d_M$ and $d_{T^*M}$. Let $\pi_M:T^*M\to M$ denote the canonical projection. Let
$H\in C^2(T^*M)$ be Tonelli: $p\mapsto H(x,p)$ is strictly convex with
positive definite Hessian and is superlinear, uniformly in $x$.  Let
$L:TM\to\R$ be its Legendre transform and put $c=c(H)$. Throughout the paper, solutions, subsolutions, and supersolutions of Hamilton--Jacobi equations are understood in the viscosity sense. Recall that $\A\subset M$ and $\widetilde{\A}\subset T^*M$ denote the projected and lifted Aubry sets, respectively; see \cite{DFIZ}.

For $\lambda>0$, let $u_\lambda$ solve \eqref{el}. We denote by $u_0$ the selected limit characterized in \cite{DFIZ}.  The family
$\{u_\lambda:0<\lambda\le1\}$ is uniformly bounded and equi-Lipschitz. A backward $u_0$-calibrated curve
$\gamma_0:(-\infty,0]\to M$ satisfies, for every $r<t\le0$,
\[
 u_0(\gamma_0(t))-u_0(\gamma_0(r))
 =
 \int_r^t\bigl(L(\gamma_0(s),\dot\gamma_0(s))+c\bigr)\,\dd s.
\]
A backward $u_\lambda$-calibrated curve
$\gamma_\lambda:(-\infty,0]\to M$ satisfies
\[
 e^{\lambda t}u_\lambda(\gamma_\lambda(t))
 -e^{\lambda r}u_\lambda(\gamma_\lambda(r))
 =
 \int_r^t e^{\lambda s}
 \bigl(L(\gamma_\lambda(s),\dot\gamma_\lambda(s))+c\bigr)\,\dd s.
\]
We denote the cotangent lifts of $\gamma_0$ and $\gamma_\lambda$ by \[Z_0(s):=(\gamma_0(s),p_0(s))\qquad \text{and} \qquad Z_\lambda(s):=(\gamma_\lambda(s),p_\lambda(s)),\]
where $p_0(s)=L_v(\gamma_0(s),\dot\gamma_0(s))$, and $p_\lambda(s)=L_v(\gamma_\lambda(s),\dot\gamma_\lambda(s))$. We also set
\[X_\lambda(x,p):=
 \bigl(H_p(x,p),-H_x(x,p)-\lambda p\bigr),
\]
so that $\dot Z_0=X_H(Z_0)$ and
$\dot Z_\lambda=X_\lambda(Z_\lambda)$, where $X_0=X_H$ is the conservative Hamiltonian vector field.
All such lifts remain in a fixed compact set $K\subset T^*M$, and their
speeds are uniformly bounded for $0<\lambda\le1$.

\smallskip
\paragraph{{\it Upper estimate.}}
\smallskip

Fix $x\in M$ and let $\gamma_0:(-\infty,0]\to M$ be a backward
$u_0$-calibrated curve with $\gamma_0(0)=x$.  Then, for almost every $s<0$,
\[
  \frac{\dd}{\dd s}u_0(\gamma_0(s))
  =L(\gamma_0(s),\dot\gamma_0(s))+c.
\]
Using $\gamma_0$ as a competitor in the discounted representation formula
and integrating by parts on $[-R,0]$, we obtain
\[
\begin{aligned}
u_\lambda(x)
&\le
\int_{-\infty}^0 e^{\lambda s}
\frac{\dd}{\dd s}u_0(\gamma_0(s))\,\dd s,\\
\int_{-R}^0e^{\lambda s}
\frac{\dd}{\dd s}u_0(\gamma_0(s))\,\dd s
&=
u_0(x)-e^{-\lambda R}u_0(\gamma_0(-R))
-\lambda\int_{-R}^0e^{\lambda s}u_0(\gamma_0(s))\,\dd s.
\end{aligned}
\]
Since $u_0$ is bounded, $e^{-\lambda R}u_0(\gamma_0(-R))\to0$. We get
\begin{equation}\label{eq:upper}
u_\lambda(x)-u_0(x)
\le
-\lambda\int_{-\infty}^0e^{\lambda s}u_0(\gamma_0(s))\,\dd s.
\end{equation}

\smallskip
\paragraph{{\it Lower estimate.}}
\smallskip

Let $\gamma_\lambda:(-\infty,0]\to M$ be a backward
$u_\lambda$-calibrated curve with $\gamma_\lambda(0)=x$.  Then
\[
u_\lambda(x)
=
\int_{-\infty}^0
e^{\lambda s}\bigl(L(\gamma_\lambda(s),\dot\gamma_\lambda(s))+c\bigr)\,\dd s.
\]
Since $u_0$ is a critical subsolution,
\[
\frac{\dd}{\dd s}u_0(\gamma_\lambda(s))
\le
L(\gamma_\lambda(s),\dot\gamma_\lambda(s))+c
\quad\text{for a.e. }s.
\]
The same integration by parts gives
\begin{equation}\label{eq:lower}
u_\lambda(x)-u_0(x)
\ge
-\lambda\int_{-\infty}^0e^{\lambda s}
u_0(\gamma_\lambda(s))\,\dd s.
\end{equation}

\smallskip

For the lower bound, \eqref{eq:lower} reduces the problem to
\[
\int_{-\infty}^0 e^{\lambda s}u_0(\gamma_\lambda(s))\,\dd s\le C.
\]
This estimate will follow in Subsection~\ref{subsec:finite-theorem-proof}
from Lemma~\ref{lem:finite-distance}. The general upper bound is based on Lemma~\ref{lem:peierls-compression}.

\subsection{A strict critical subsolution}

By  
\cite{FathiSiconolfi}, there exists a critical subsolution $w\in C^1(M)$
which is strict outside the projected Aubry set, i.e.,
\begin{equation}\label{eq:strict}
H(x,Dw(x))\le c\quad\text{on }M,
\qquad
H(x,Dw(x))<c\quad\text{on }M\setminus\A.
\end{equation}
It is also known that the Aubry set has the graph property $\widetilde{\A}\subset\{(x,Dw(x)):x\in M\}$. Equivalently,
for $x\in\A$, $Dw(x)$ is the unique momentum above $x$ in the lifted
Aubry set.
For $(x,p)\in T^*M$, define
\begin{equation}\label{eq:q-def}
q(x,p)
=
L\bigl(x,H_p(x,p)\bigr)+c
-\bigl\langle Dw(x),H_p(x,p)\bigr\rangle.
\end{equation}

\begin{lemma}\label{lem:q}
The function $q$ is continuous, nonnegative, and $q^{-1}(0)=\widetilde{\A}$.
\end{lemma}

\begin{proof}
Adding and subtracting $H(x,Dw(x))$ gives
\[
\begin{aligned}
q(x,p)
&=
\Bigl[
L\bigl(x,H_p(x,p)\bigr)+H(x,Dw(x))
-\langle Dw(x),H_p(x,p)\rangle
\Bigr]\\
&\quad+
\bigl[c-H(x,Dw(x))\bigr].
\end{aligned}
\]
Both brackets are nonnegative: the first by Fenchel's inequality and the
second by \eqref{eq:strict}.  We get $q\ge0$.

Suppose
first that $q(x,p)=0$.  Since the two nonnegative terms above have sum zero,
both vanish.  The second one and the strictness in \eqref{eq:strict} imply
$x\in\A$.  Equality in Fenchel's inequality for the first one gives $H_p(x,p)=H_p(x,Dw(x))$.
For fixed $x$, strict convexity of $H(x,\cdot)$ implies that
$p\mapsto H_p(x,p)$ is injective. Hence, $p=Dw(x)$.  By the graph
property of the Aubry set, $(x,Dw(x))\in\widetilde{\A}$.  Thus, 
$q^{-1}(0)\subset\widetilde{\A}$.

Conversely, let $(x,p)\in\widetilde{\A}$.  Then $x\in\A$ and the
graph property gives $p=Dw(x)$.  Moreover
$H(x,Dw(x))=c$ on the Aubry set.  Therefore equality holds in Fenchel's
inequality at the pair $(Dw(x),H_p(x,Dw(x)))$, and both nonnegative terms in
the decomposition of $q$ vanish.  Therefore, $q(x,p)=0$, proving
$\widetilde{\A}\subset q^{-1}(0)$.

Finally, $q$ is continuous because $w\in C^1$ and $H$, $L$ are of class $C^2$.
\end{proof}

Along $Z_0$, calibration
and the chain rule give
\[
q(Z_0(s))
=\frac{\dd}{\dd s}(u_0-w)(\gamma_0(s))
\quad\text{for a.e. }s<0.
\]
Thus, for every $R>0$,
\[
\int_{-R}^0 q(Z_0(s))\,\dd s
=
(u_0-w)(x)-(u_0-w)(\gamma_0(-R)),
\]
and therefore
\begin{equation}\label{eq:q0}
\int_{-\infty}^0 q(Z_0(s))\,\dd s\le C.
\end{equation}

Along $Z_\lambda$,
\[
q(Z_\lambda(s))
=L(\gamma_\lambda(s),\dot\gamma_\lambda(s))+c
-\frac{\dd}{\dd s}w(\gamma_\lambda(s))
\quad\text{for a.e. }s<0.
\]
For $R>0$, since $w\in C^1(M)$ and
$\gamma_\lambda\in AC([-R,0];M)$, the composition
$w\circ\gamma_\lambda$ is absolutely continuous and
\[
\frac{\dd}{\dd s}w(\gamma_\lambda(s))
=\langle Dw(\gamma_\lambda(s)),\dot\gamma_\lambda(s)\rangle
\quad\text{for a.e. }s\in[-R,0].
\]
Integrating by parts we obtain
\[
\begin{aligned}
\int_{-R}^0 e^{\lambda s}q(Z_\lambda(s))\,\dd s
&=
\int_{-R}^0e^{\lambda s}
\bigl(L(\gamma_\lambda(s),\dot\gamma_\lambda(s))+c\bigr)\,\dd s\\
&\quad-w(x)+e^{-\lambda R}w(\gamma_\lambda(-R))+\lambda\int_{-R}^0e^{\lambda s}w(\gamma_\lambda(s))\,\dd s.
\end{aligned}
\]
Letting $R\to\infty$, the boundary term tends to zero because $w$ is
bounded, while the first integral converges to $u_\lambda(x)$ by the
discounted representation formula.  We obtain
\[
\int_{-\infty}^0 e^{\lambda s}q(Z_\lambda(s))\,\dd s
=u_\lambda(x)-w(x)
+\lambda\int_{-\infty}^0e^{\lambda s}w(\gamma_\lambda(s))\,\dd s.
\]
Since $u_\lambda$ and $w$ are uniformly bounded,
\begin{equation}\label{eq:qlambda}
\int_{-\infty}^0 e^{\lambda s}q(Z_\lambda(s))\,\dd s\le C.
\end{equation}

For later use, if $U$ is any fixed neighborhood of
$\widetilde{\A}$, then continuity of $q$, Lemma~\ref{lem:q}, and
compactness of $K$ give $q\ge\delta_U>0$ on $K\setminus U$.

\section{Local estimates near hyperbolic equilibria and periodic orbits}\label{sec:hyperbolic}

This section proves local estimates for orbit segments that remain near a fixed hyperbolic component of the Aubry set. Their integrated form is used in Lemma \ref{lem:finite-distance}, while the exponential decay property is used in Lemma \ref{lem:peierls-compression}.

The local
estimates below are stated for arbitrary real time intervals. When they are
applied to a calibrated curve, the interval is a subinterval
of $(-\infty,0]$.

\subsection{A hyperbolic equilibrium}

Let $z_0\in\widetilde{\A}$ be a hyperbolic equilibrium of $X_H$.

\begin{lemma}\label{lem:eq-local}
There exist $\rho_0,\mu,C,\lambda_0>0$ such that the following holds.  Let
$0<\rho\le\rho_0$, $0\le\lambda\le\lambda_0$, and let $a<b$ be real
numbers.  If $Z^\lambda:[a,b]\to T^*M$ is an orbit of $X_\lambda$ satisfying $Z^\lambda(s)\in B_{2\rho}(z_0)$
for every $s\in[a,b]$,
then for every $s\in[a,b]$,
\begin{equation}\label{eq:eq-point}
 d_{T^*M}(Z^\lambda(s),z_0)
 \le
 C\rho\bigl(e^{-\mu(s-a)}+e^{-\mu(b-s)}\bigr)+C\lambda.
\end{equation}
\end{lemma}

\begin{proof}
Choose local coordinates centered at $z_0$ and write $y(s)$ for the
coordinate representation of $Z^\lambda(s)$.  Since $z_0$ is an equilibrium of the
conservative flow, $A:=DX_H(z_0)$ is hyperbolic.  Let $\R^{2n}=E^s\oplus E^u$ be the decomposition into the invariant stable and unstable subspaces of
$A$, and let $P^s,P^u$ be the corresponding projections.  The exponential
decay rates on $E^s$ in forward time and on $E^u$ in backward time need not
coincide.  Choose $\mu>0$ so that $2\mu$ is smaller than both rates.  Then
there is $C_0\ge1$ such that
\[
 \|e^{A\tau}P^s\|\le C_0e^{-2\mu\tau}
 \quad\text{for every }\tau\ge0,
 \qquad
 \|e^{A\tau}P^u\|\le C_0e^{2\mu\tau}
 \quad\text{for every }\tau\le0.
\]

We next estimate the nonlinear and discounted perturbations.  On a fixed
small coordinate neighborhood of $z_0$, the momentum variable is bounded,
and
\[
 X_\lambda(x,p)-X_H(x,p)=(0,-\lambda p).
\]
Hence $|X_\lambda-X_H|\le C\lambda$ there.  Moreover, since $X_H$ is
differentiable at $z_0$, for every $\varepsilon>0$ there is a sufficiently
small neighborhood of $z_0$ such that
\[
 |X_H(z_0+y)-Ay|\le\varepsilon|y|.
\]
Consequently, once the size of the neighborhood is chosen sufficiently
small, the orbit equation can be written as
\begin{equation}\label{eq:eq-remainder}
 \dot y=Ay+r(s),
 \qquad
 |r(s)|\le\varepsilon |y(s)|+C\lambda.
\end{equation}
Here $\varepsilon>0$ will be fixed at the end of the argument. The constant
$C$ is independent of $a,b,\rho$ and $\lambda$.

We estimate the stable component from the left endpoint and the unstable
component from the right endpoint.  For the equation
$\dot y=Ay+r(s)$, the usual solution formula gives
\[
 P^sy(s)
 =e^{A(s-a)}P^sy(a)
 +\int_a^s e^{A(s-\tau)}P^sr(\tau)\,\dd\tau,
\]
and, writing the same formula backwards from $b$,
\[
 P^uy(s)
 =e^{A(s-b)}P^uy(b)
 -\int_s^b e^{A(s-\tau)}P^ur(\tau)\,\dd\tau.
\]
Since the whole orbit segment lies in $B_{2\rho}(z_0)$, the endpoint values
satisfy $|y(a)|+|y(b)|\le C\rho$.  Therefore,
\begin{equation}\label{eq:eq-stable-unstable-bounds}
\begin{aligned}
 |P^sy(s)|
 &\le C\rho e^{-2\mu(s-a)}
 +C\int_a^s e^{-2\mu(s-\tau)}
       \bigl(\varepsilon|y(\tau)|+\lambda\bigr)\,\dd\tau,\\
 |P^uy(s)|
 &\le C\rho e^{-2\mu(b-s)}
 +C\int_s^b e^{-2\mu(\tau-s)}
       \bigl(\varepsilon|y(\tau)|+\lambda\bigr)\,\dd\tau.
\end{aligned}
\end{equation}

Set
\[
 h(s):=e^{-\mu(s-a)}+e^{-\mu(b-s)},
 \qquad
 N:=\sup_{a\le s\le b}
 \frac{|y(s)|}{\rho h(s)+\lambda}.
\]
Then
\[
 |y(\tau)|\le N\bigl(\rho h(\tau)+\lambda\bigr).
\]
We give the details for the stable component. The unstable estimate is
analogous.  Substituting the preceding bound into the first inequality in
\eqref{eq:eq-stable-unstable-bounds} gives
\[
\begin{aligned}
 |P^sy(s)|
 \le C\rho e^{-2\mu(s-a)}
&+C\varepsilon N\rho
 \int_a^s e^{-2\mu(s-\tau)}e^{-\mu(\tau-a)}\,\dd\tau\\
 &+C\varepsilon N\rho
 \int_a^s e^{-2\mu(s-\tau)}e^{-\mu(b-\tau)}\,\dd\tau\\
 &+C(\varepsilon N+1)\lambda
 \int_a^s e^{-2\mu(s-\tau)}\,\dd\tau.
\end{aligned}
\]
The three integrals satisfy
\[
 \int_a^s e^{-2\mu(s-\tau)}e^{-\mu(\tau-a)}\,\dd\tau
 \le \frac1\mu e^{-\mu(s-a)},
\]
\[
 \int_a^s e^{-2\mu(s-\tau)}e^{-\mu(b-\tau)}\,\dd\tau
 \le \frac1{3\mu}e^{-\mu(b-s)},
\]
and
\[
 \int_a^s e^{-2\mu(s-\tau)}\,\dd\tau\le\frac1{2\mu}.
\]
Since $e^{-2\mu(s-a)}\le e^{-\mu(s-a)}\le h(s)$, it follows that
\[
 |P^sy(s)|
 \le C\bigl(\rho h(s)+\lambda\bigr)
 +C\varepsilon N\bigl(\rho h(s)+\lambda\bigr).
\]
The same calculation, using the second inequality in
\eqref{eq:eq-stable-unstable-bounds} and the right endpoint $b$, gives
\[
 |P^uy(s)|
 \le C\bigl(\rho h(s)+\lambda\bigr)
 +C\varepsilon N\bigl(\rho h(s)+\lambda\bigr).
\]

Adding these two estimates, there is
a constant $C_1$ independent of
$\varepsilon,a,b,\rho,\lambda$ such that
\[
 |y(s)|
 \le
 C_1\bigl(\rho h(s)+\lambda\bigr)
 +C_1\varepsilon N\bigl(\rho h(s)+\lambda\bigr).
\]
Dividing by
$\rho h(s)+\lambda$ and taking the supremum gives
\[
 N\le C_1+C_1\varepsilon N.
\]
We now choose $\varepsilon>0$ so small that $C_1\varepsilon\le1/2$, and
then choose $\rho_0>0$ sufficiently small so that
\eqref{eq:eq-remainder} holds with this value of $\varepsilon$ whenever
$|y|\le2\rho_0$.  After fixing such a neighborhood, choose
$\lambda_0>0$ sufficiently small.  Then
\[
 N\le C_1+\frac12N
 \qquad\Longrightarrow\qquad
 N\le2C_1.
\]
This proves \eqref{eq:eq-point}.
\end{proof}

Two consequences will be used below. For $\lambda=0$,
\[
  \int_a^b d_{T^*M}(Z^0(s),z_0)\,\dd s\le C\rho.
\]
If $Z^\lambda$ is defined on $(-\infty,b]$ and remains in
$B_{2\rho}(z_0)$, letting $a\to-\infty$ gives
\[
  d_{T^*M}(Z^\lambda(s),z_0)
  \le C\rho e^{-\mu(b-s)}+C\lambda
  \qquad \text{for every }s\le b.
\]

\subsection{A hyperbolic periodic orbit}

Let $\Gamma\subset\widetilde{\A}$ be a periodic orbit of the conservative
Hamiltonian flow with prime period $T$.  We assume that it is hyperbolic in
the critical energy level $H^{-1}(c)$: after removing the flow direction,
the linear flow is hyperbolic in the transverse coordinate $y$ defined below.

Besides the flow direction, there is an energy direction which is not
controlled by this hyperbolicity assumption.  For a $u_\lambda$-calibrated
curve $\gamma_\lambda$, the Hamilton--Jacobi equation controls that direction directly.  Let
$Z_\lambda(s)=(\gamma_\lambda(s),p_\lambda(s))$.  At almost every interior time,
$p_\lambda(s)=Du_\lambda(\gamma_\lambda(s))$, and therefore
\[
 H(Z_\lambda(s))-c=-\lambda u_\lambda(\gamma_\lambda(s)).
\]
Since $u_\lambda$ is uniformly bounded,
we obtain for every time on the curve
\begin{equation}\label{eq:energy}
 |H(Z_\lambda(s))-c|\le C\lambda.
\end{equation}

The next lemma contains the two finite-interval estimates needed later.  The
first is used for $u_\lambda$-calibrated curves in
Lemma~\ref{lem:finite-distance}.  In Lemma~\ref{lem:peierls-compression},
the lifts of the curves $\gamma_n$ minimizing $h_{T_n}(a,x)$ are conservative orbits with
energy $c+e_n$, where $e_n\to0$. Part (ii) is used for these orbit segments.

\begin{lemma}\label{lem:per-local}
There exist $\rho_0,\mu,C,\lambda_0,e_0>0$ such that the following holds for
$0<\rho\le\rho_0$ and every real interval $[a,b]$.

\begin{enumerate}[label=\textnormal{(\roman*)}]
\item If $0<\lambda\le\lambda_0$ and
$Z_\lambda:[a,b]\to T^*M$ is the lift of a $u_\lambda$-calibrated curve with
\[
 d_{T^*M}(Z_\lambda(s),\Gamma)<2\rho
 \qquad\text{for every }a\le s\le b,
\]
then
\begin{equation}\label{eq:per-point}
 d_{T^*M}(Z_\lambda(s),\Gamma)
 \le C\rho\bigl(e^{-\mu(s-a)}+e^{-\mu(b-s)}\bigr)+C\lambda
 \qquad\text{for every }a\le s\le b.
\end{equation}

\item If $Z^0:[a,b]\to T^*M$ is a conservative orbit satisfying
\[
 \dot Z^0=X_H(Z^0),
 \qquad
 H(Z^0(s))\equiv c+e,
 \qquad |e|\le e_0,
\]
and
\[
 d_{T^*M}(Z^0(s),\Gamma)<2\rho
 \qquad\text{for every }a\le s\le b,
\]
then
\begin{equation}\label{eq:per-energy-point}
 d_{T^*M}(Z^0(s),\Gamma)
 \le C\rho\bigl(e^{-\mu(s-a)}+e^{-\mu(b-s)}\bigr)+C|e|
 \quad\text{for every }a\le s\le b.
\end{equation}
\end{enumerate}
\end{lemma}

\begin{proof}
Parametrize $\Gamma$ by the conservative flow $X_H$.  Since $\Gamma$ is a
nonstationary periodic orbit, $X_H(z)\neq0$ for every $z\in\Gamma$.  By
nondegeneracy of the symplectic form, $dH(z)=0$ if and only if
$X_H(z)=0$. It follows that $dH\neq0$ on $\Gamma$.  By the implicit function theorem,
after shrinking a neighborhood $\mathcal V$ of $\Gamma$, $\Sigma:=H^{-1}(c)\cap\mathcal V$ is a $C^2$ hypersurface.  Since $X_H\in C^1$, the periodic parametrization of
$\Gamma$ is $C^2$.  The hypersurface $\Sigma$ is cooriented by $dH$, while
the flow orients $\Gamma$. Hence, the normal bundle of $\Gamma$ in $\Sigma$
is orientable and therefore trivial over the circle.  Fix a $C^1$
trivialization of this normal bundle.  Using the flow phase $\theta$,
coordinates $y$ transverse to the orbit within $\Sigma$, and $E=H-c$ as the
remaining coordinate, a standard tubular construction gives a sufficiently
small relatively compact neighborhood with adapted coordinates
\[
 (\theta,E,y)\in(\R/T\mathbb Z)\times\R\times\R^{2n-2}.
\]
In these coordinates $\Gamma=\{E=0,\ y=0\}$, and the phase is chosen so that
\[
 \frac{\dd\theta}{\dd s}=1
 \qquad\text{on }\Gamma.
\]
In what follows, a dot denotes differentiation with respect to the original
time variable $s$ until we explicitly change the independent variable from
$s$ to $\theta$.

We first write the conservative vector field in the tubular coordinates.  Let
$G(\theta,E,y)$ and $F(\theta,E,y)$ denote respectively its $\theta$- and
$y$-components.  Along the periodic orbit,
\[
 G(\theta,0,0)=1,
 \qquad
 F(\theta,0,0)=0.
\]
Set $A(\theta):=D_yF(\theta,0,0)$. Because $\dot\theta=1$ on $\Gamma$, the variational equation in the
$y$-directions inside the critical energy surface $E=0$ is
\begin{equation}\label{eq:per-linear}
 \frac{\dd y}{\dd\theta}=A(\theta)y.
\end{equation}
For $\sigma,\theta\in\R$ and $v\in\R^{2n-2}$, let
$U(\theta,\sigma)v$ denote the value at phase $\theta$ of the solution of
\eqref{eq:per-linear} with initial value $y(\sigma)=v$.  On the critical
energy surface $E=0$, the variable $y$ describes precisely the directions
transverse to the flow.  Since $\Gamma$ is hyperbolic in $H^{-1}(c)$, these
$y$-directions split into invariant stable and unstable subspaces.  Let
$P^s(\theta)$ and $P^u(\theta)$ denote the corresponding projections in the
$y$-coordinates.  After choosing $\mu>0$ smaller than the stable and unstable
exponential rates, there is $C_0\ge1$ such that
\[
 \|U(\theta,\sigma)P^s(\sigma)\|
 \le C_0e^{-2\mu(\theta-\sigma)}
 \qquad\text{for every }\theta\ge\sigma,
\]
and
\[
 \|U(\theta,\sigma)P^u(\sigma)\|
 \le C_0e^{-2\mu(\sigma-\theta)}
 \qquad\text{for every }\theta\le\sigma.
\]

Since $H\in C^2$, the Hamiltonian vector field $X_H$ is $C^1$.  With the
adapted coordinates above, $G$ and $F$ are $C^1$ in $(E,y)$, and
$\partial_EG$, $D_yG$, $\partial_EF$, and $D_yF$ are continuous in
$(\theta,E,y)$.  Hence Taylor expansion in $(E,y)$ at $(0,0)$ is uniform with
respect to the phase:
\begin{equation}\label{eq:per-transverse-expansion}
\begin{aligned}
 G(\theta,E,y)
 &=1+\partial_EG(\theta,0,0)E
   +D_yG(\theta,0,0)y+R_\theta(\theta,E,y),\\
 F(\theta,E,y)
 &=A(\theta)y+\partial_EF(\theta,0,0)E
   +R_y(\theta,E,y),
\end{aligned}
\end{equation}
where there is a modulus $\omega(r)\to0$ as $r\to0$ such that
\begin{equation}\label{eq:per-remainders}
 |R_\theta(\theta,E,y)|+|R_y(\theta,E,y)|
 \le \omega(r)(|E|+|y|)
\end{equation}
whenever $|E|+|y|\le r$.  The uniformity in $\theta$ follows from the
continuity of the displayed derivatives and compactness of the phase circle.

We treat $Z_\lambda$ and $Z^0$ in parallel.  Put
\[
 \eta=
 \begin{cases}
   \lambda,&\text{for }Z_\lambda,\\
   |e|,&\text{for }Z^0\text{ with }H(Z^0)=c+e.
 \end{cases}
\]
For $Z_\lambda$, \eqref{eq:energy} implies $|E(s)|\le C\eta$ for every
$s\in[a,b]$.  For $Z^0$, we have $E(s)\equiv e$, so the same bound
holds.

Since $\Gamma$ is compact, after fixing the tubular neighborhood the momentum
variable $p$ is uniformly bounded there.  We have
\[
 X_\lambda-X_H=(0,-\lambda p),
 \qquad
 |X_\lambda-X_H|\le C\lambda,
\]
and the perturbation remains $O(\lambda)$ after passing to the fixed tubular
coordinates.  If the orbit stays in the $2\rho$-tube, then
$|E(s)|+|y(s)|\le C\rho$.  Using also $|E(s)|\le C\eta$, the expansions
\eqref{eq:per-transverse-expansion}--\eqref{eq:per-remainders} give, after
enlarging the modulus if necessary,
\[
 |\dot\theta(s)-1|\le C\bigl(|y(s)|+\eta\bigr),
 \qquad
 |\dot y(s)-A(\theta(s))y(s)|
 \le \omega_1(\rho)|y(s)|+C\eta,
\]
where $\omega_1(\rho)\to0$ as $\rho\to0$.

Choose first $\rho_0>0$ and then $\lambda_0,e_0>0$ sufficiently small so
that
\begin{equation}\label{eq:per-theta-monotone}
 \frac12\le\dot\theta(s)\le\frac32
 \qquad\text{for every }s\in[a,b].
\end{equation}
This shows that $\theta$ is strictly increasing and may be used as the independent
variable.  Using the two estimates above and the boundedness of $A(\theta)$,
we obtain
\[
\begin{aligned}
 \left|\frac{\dd y}{\dd\theta}-A(\theta)y\right|
 &\le
 \frac{|\dot y-A(\theta)y|}{|\dot\theta|}
 +\left|\frac1{\dot\theta}-1\right|\,|A(\theta)y|\\
 &\le 2\omega_1(\rho)|y|+C\eta+C|1-\dot\theta|\,|y|\\
 &\le \bigl(2\omega_1(\rho)+C\rho\bigr)|y|+C\eta,
\end{aligned}
\]
where the last inequality uses $|1-\dot\theta|\le C(|y|+\eta)$ and
$|y|\le C\rho$.  Consequently, for every prescribed $\varepsilon>0$, by
decreasing $\rho_0$ if necessary we obtain
\begin{equation}\label{eq:per-error}
 \left|\frac{\dd y}{\dd\theta}-A(\theta)y\right|
 \le\varepsilon|y|+C\eta.
\end{equation}
The constant $C$ is independent of the prescribed $\varepsilon$.

Choose a continuous lift of the phase on $[a,b]$ and set $\alpha:=\theta(a)$, $\beta:=\theta(b)$.
Since the endpoints lie in the $2\rho$-tube,
$|y(\alpha)|+|y(\beta)|\le C\rho$.  Write
\[
 \frac{\dd y}{\dd\theta}=A(\theta)y+r(\theta),
 \qquad
 |r(\theta)|\le\varepsilon|y(\theta)|+C\eta.
\]
Set
\[
 h_\theta(\theta)
 :=e^{-\mu(\theta-\alpha)}+e^{-\mu(\beta-\theta)},
 \qquad
 N:=\sup_{\alpha\le\theta\le\beta}
 \frac{|y(\theta)|}{\rho h_\theta(\theta)+\eta}.
\]
The denominator is strictly positive on the compact interval
$[\alpha,\beta]$, so $N<\infty$.  Applying the same estimates as in the proof of Lemma \ref{lem:eq-local}, there is a constant $C_1$, independent of
$\varepsilon,a,b,\rho,\lambda$ and $e$, such that
\[
 |y(\theta)|
 \le
 C_1\bigl(\rho h_\theta(\theta)+\eta\bigr)
 +C_1\varepsilon N\bigl(\rho h_\theta(\theta)+\eta\bigr).
\]
Dividing by $\rho h_\theta(\theta)+\eta$ and taking the supremum gives
\[
 N\le C_1+C_1\varepsilon N.
\]
Since $C_1$ is independent of $\varepsilon$, we now choose
$\varepsilon>0$ so small that $C_1\varepsilon\le1/2$, and then fix
$\rho_0,\lambda_0,$ and $e_0$ small enough that \eqref{eq:per-error} holds
with this choice of $\varepsilon$.  Then
\[
 N\le C_1+\frac12N
 \qquad\Longrightarrow \qquad
 N\le2C_1.
\]
Consequently,
\begin{equation}\label{eq:per-phase-estimate}
 |y(\theta)|
 \le
 C\rho\bigl(e^{-\mu(\theta-\alpha)}
             +e^{-\mu(\beta-\theta)}\bigr)+C\eta.
\end{equation}
Together with $|E|\le C\eta$ and equivalence of the tubular coordinates
with the Riemannian distance, this controls the distance to $\Gamma$ in the
phase variable.

We finally return from phase to the original time variable.  By
\eqref{eq:per-theta-monotone}, for every
$s\in[a,b]$,
\[
 \theta(s)-\alpha
 =\int_a^s\dot\theta(\tau)\,\dd\tau
 \ge\frac12(s-a),
\]
and
\[
 \beta-\theta(s)
 =\int_s^b\dot\theta(\tau)\,\dd\tau
 \ge\frac12(b-s).
\]
By \eqref{eq:per-phase-estimate}, after decreasing $\mu$ by a fixed
factor and renaming it, we get \eqref{eq:per-point}
and \eqref{eq:per-energy-point}.
\end{proof}

When $e=0$, integration of the conservative estimate gives
\[
\int_a^b d_{T^*M}(Z^0(s),\Gamma)\,\dd s \le C\rho.
\]
If $Z_\lambda$ is defined on $(-\infty,b]$ and remains
in the $2\rho$-neighborhood of $\Gamma$, then letting the left endpoint tend to
$-\infty$ in \eqref{eq:per-point} gives
\[
d_{T^*M}(Z_\lambda(s),\Gamma)
\le C\rho e^{-\mu(b-s)}+C\lambda
\qquad \text{for every }s\le b.
\]
For an orbit $Z^0$ of energy $c$, the same argument applied to
\eqref{eq:per-energy-point} gives
\[
d_{T^*M}(Z^0(s),\Gamma)
\le C\rho e^{-\mu(b-s)}
\qquad \text{for every }s\le b.
\]

\section{Proof of the main theorem}\label{sec:finite-proof}

Throughout this section,
we assume (A).
We write $\A_i:=\pi_M(\Gamma_i)$.
The projected Mather measure associated with $\Gamma_i$ is denoted by
$\mu_i$, and $m_i:=\int_Mu_0\,\dd\mu_i\le0$.

\subsection{Active components and the selected solution}

For $t>0$, let
\[
 h_t(y,x)
 :=
 \inf_{\substack{\gamma(0)=y\\ \gamma(t)=x}}
 \int_0^t\bigl(L(\gamma(s),\dot\gamma(s))+c\bigr)\,\dd s
\]
and let
\[
 h(y,x):=\liminf_{t\to\infty}h_t(y,x)
\]
be the Peierls barrier.  The following lemma identifies, for each $x\in M$,
an active component $\Gamma_i$ from which the Peierls barrier represents $u_0(x)$.

\begin{lemma}\label{lem:active-representation}
Recall that $\A_i=\pi_M(\Gamma_i)$, and let $I_*:=\{i\in\{1,\dots,N\}:m_i=0\}$.
Then $I_*\neq\emptyset$.  Moreover, for every $x\in M$ there exists
$i\in I_*$ such that, for every $a\in \A_i$,
\begin{equation}\label{eq:active-representation}
 u_0(x)=u_0(a)+h(a,x).
\end{equation}
\end{lemma}

\begin{proof}
Every Mather measure is supported on
$\widetilde{\A}=\Gamma_1\sqcup\cdots\sqcup\Gamma_N$.  Since the $\Gamma_i$ are
pairwise disjoint invariant sets and each equilibrium or periodic orbit
carries a unique invariant probability, every projected Mather measure is a
convex combination of the measures $\mu_i$.

By Proposition~1.3 of \cite{DFIZ}, the selected solution satisfies
$m_i\le0$ for every $i$ and is represented by minimizing $\mu\to\int_M h(y,x)\,\dd\mu(y)$
over all projected Mather measures.  In the present finite setting every such
measure is a convex combination of $\mu_1,\dots,\mu_N$, and the displayed
functional is affine in $\mu$.  Its minimum over this convex hull is therefore
attained at one of the measures $\mu_i$.  Thus,
\begin{equation}\label{eq:DFIZ-finite}
 u_0(x)=\min_{1\le i\le N}\int_M h(y,x)\,\dd\mu_i(y).
\end{equation}

For every $i$, domination of $u_0$ gives
\[
 h(y,x)\ge u_0(x)-u_0(y).
\]
After integration with respect to $\mu_i$,
\begin{equation}\label{eq:active-domination-average}
 \int_M h(y,x)\,\dd\mu_i(y)\ge u_0(x)-m_i.
\end{equation}
Since the minimum in \eqref{eq:DFIZ-finite} is exactly $u_0(x)$, every
index realizing that minimum must satisfy $m_i=0$: if $m_i<0$, then
\eqref{eq:active-domination-average} is strictly larger than $u_0(x)$.
In particular $I_*\neq\emptyset$, and for every $x$ at least one minimizing
index belongs to $I_*$.

Fix such an $i$ and any $a\in \A_i$.  The measure $\mu_i$ is supported in
$\A_i$.  If $y\in \A_i$, then $a$ and $y$ belong to the same static class, so
\[
 h(a,y)+h(y,a)=0.
\]
Domination gives
\[
 u_0(y)-u_0(a)\le h(a,y),
 \qquad
 u_0(a)-u_0(y)\le h(y,a).
\]
Since the two left-hand sides and the two right-hand sides both sum to
zero, both inequalities are equalities. We have
\[
 h(y,a)=u_0(a)-u_0(y),
 \qquad
 h(a,y)=u_0(y)-u_0(a).
\]
The triangle inequality gives
\[
 h(y,x)\le h(y,a)+h(a,x)
 =u_0(a)-u_0(y)+h(a,x),
\]
and also
\[
 h(a,x)\le h(a,y)+h(y,x)
 =u_0(y)-u_0(a)+h(y,x).
\]
Consequently, for every $y\in \A_i$,
\[
h(y,x)=u_0(a)-u_0(y)+h(a,x).
\]
Integrating against $\mu_i$ and using $m_i=0$ gives
\[
 \int_M h(y,x)\,\dd\mu_i(y)=u_0(a)+h(a,x).
\]
The left-hand side realizes the minimum in \eqref{eq:DFIZ-finite}, so
\eqref{eq:active-representation} follows.
\end{proof}

\subsection{Global distance estimates}

The following lemma gives the global distance estimates used in
Subsection~\ref{subsec:finite-theorem-proof}.

For $r>0$, set
\[
 U_r^i:=\{z\in T^*M:d_{T^*M}(z,\Gamma_i)<r\},
 \qquad
 \mathcal U_r:=\bigcup_{i=1}^N U_r^i.
\]

\begin{lemma}\label{lem:finite-distance}
There is $C>0$ such that every lift of a $u_0$-calibrated curve
$Z_0:(-\infty,0]\to T^*M$ satisfies
\begin{equation}\label{eq:finite-global0}
 \int_{-\infty}^0 d_{T^*M}(Z_0(s),\widetilde{\A})\,\dd s\le C,
\end{equation}
and every lift of a $u_\lambda$-calibrated curve
$Z_\lambda:(-\infty,0]\to T^*M$ satisfies
\begin{equation}\label{eq:finite-global-lambda}
 \int_{-\infty}^0 e^{\lambda s}
 d_{T^*M}(Z_\lambda(s),\widetilde{\A})\,\dd s\le C
\end{equation}
for all sufficiently small $\lambda>0$.
\end{lemma}

\begin{proof}
Because the number of components $\Gamma_i$ is finite, choose $\rho>0$
so small that the neighborhoods $U_{2\rho}^i$, $i=1,\dots,N$, are
pairwise disjoint and all the local estimates of
Section~\ref{sec:hyperbolic} hold with common constants.
On the compact set \(K\) containing all lifts of $u_0$- and $u_\lambda$-calibrated curves, Lemma~\ref{lem:q} gives
\begin{equation}\label{eq:finite-qdelta}
 q\ge\delta>0
 \qquad\text{on }K\setminus\mathcal U_\rho.
\end{equation}
We also fix \(V>0\) such that
\[
 |\dot Z_0(s)|\le V,
 \qquad
 |\dot Z_\lambda(s)|\le V.
\]

We first consider the lift of a $u_0$-calibrated curve \(Z_0\).  Set
\[
 E_\rho:=\{s\le0:Z_0(s)\notin\mathcal U_\rho\}.
\]
By \eqref{eq:finite-qdelta}, \(q(Z_0(s))\ge\delta\) for \(s\in E_\rho\).
By \eqref{eq:q0}, we have
\[
 \delta |E_\rho|
 \le \int_{E_\rho}q(Z_0(s))\,\dd s
 \le \int_{-\infty}^0q(Z_0(s))\,\dd s
 \le C,
\]
and therefore
\begin{equation}\label{eq:finite-outside-time}
 |E_\rho|\le C.
\end{equation}

Consider the connected components \(I_j\) of
\[
 \{s\le0:Z_0(s)\in\mathcal U_{2\rho}\}
\]
which actually reach the inner neighborhood, i.e., $Z_0(I_j)\cap\mathcal U_\rho\neq\emptyset$. Since the sets \(U_{2\rho}^i\) are pairwise disjoint, each \(I_j\) is
associated with a unique index \(i(j)\) and $Z_0(I_j)\subset U_{2\rho}^{i(j)}$.
Let \(a_j\) and \(b_j\) denote the left and right endpoints of \(I_j\),
allowing \(a_j=-\infty\) and \(b_j=0\).  If \(b_j<0\), continuity and
maximality give
\[
 d_{T^*M}(Z_0(b_j),\Gamma_{i(j)})=2\rho.
\]
Because \(I_j\) reaches \(U_\rho^{i(j)}\), the set
\[
 \{s\in I_j:d_{T^*M}(Z_0(s),\Gamma_{i(j)})\le\rho\}
\]
is nonempty.  Define
\[
 \sigma_j:=\sup\{s\in I_j:
 d_{T^*M}(Z_0(s),\Gamma_{i(j)})\le\rho\}.
\]
Then \(\sigma_j<b_j\),
\[
 d_{T^*M}(Z_0(\sigma_j),\Gamma_{i(j)})=\rho,
\]
and
\[
 Z_0(s)\notin\mathcal U_\rho
 \qquad\text{for every }s\in(\sigma_j,b_j].
\]
Thus \([\sigma_j,b_j]\) is the last crossing, before the orbit leaves
\(U_{2\rho}^{i(j)}\), from distance \(\rho\) to distance \(2\rho\) from
\(\Gamma_{i(j)}\).  Since the velocity of $Z_0$ is bounded by \(V\),
\[
 \rho
 \le d_{T^*M}(Z_0(\sigma_j),Z_0(b_j))
 \le V(b_j-\sigma_j),
\]
so
\[
 b_j-\sigma_j\ge\frac{\rho}{V}.
\]
The intervals \([\sigma_j,b_j]\) are disjoint for different \(j\), and
\eqref{eq:finite-qdelta} holds on them.  Therefore
\[
 \#\{j:b_j<0\}\,\frac{\delta\rho}{V}
 \le
 \sum_{j:b_j<0}\int_{\sigma_j}^{b_j}q(Z_0(s))\,\dd s
 \le
 \int_{-\infty}^0q(Z_0(s))\,\dd s
 \le C.
\]
There is at most one further interval $I_j$ with \(b_j=0\).  Hence the
number of intervals $I_j$ is uniformly bounded.

We now estimate the distance integral on each interval $I_j$.  Lemmas~\ref{lem:eq-local} and \ref{lem:per-local} require the orbit segment to remain in $\mathcal U_{2\rho}$.  A
finite endpoint of $I_j$ may satisfy
$d_{T^*M}(Z_0,\Gamma_{i(j)})=2\rho$. In that case we first apply the estimate
on $[a_j+\varepsilon,b_j-\varepsilon]$, using the analogous one-sided
interval when only one endpoint is finite, and then let
$\varepsilon\downarrow0$.  If its left endpoint is finite,
Lemma~\ref{lem:eq-local} with \(\lambda=0\), or
Lemma~\ref{lem:per-local}(ii) with \(e=0\), integrated over \(I_j\), gives
\[
 \int_{I_j}d_{T^*M}(Z_0(s),\widetilde{\A})\,\dd s\le C\rho.
\]
If \(a_j=-\infty\), the corresponding left half-line estimate gives the same
bound.  Since $\#\{I_j\}$ is uniformly bounded, their total
contribution is bounded.

It remains to estimate the times which do not belong to any interval $I_j$.
Set
\[
 \mathcal R_0:=(-\infty,0]\setminus\bigcup_j I_j.
\]
By the definition of \(I_j\), $\mathcal R_0\subset E_\rho$.
By \eqref{eq:finite-outside-time}, \(|\mathcal R_0|\le C\).  Since \(K\) is
compact,
\[
 \sup_{z\in K}d_{T^*M}(z,\widetilde{\A})<\infty,
\]
which implies
\[
 \int_{\mathcal R_0}d_{T^*M}(Z_0(s),\widetilde{\A})\,\dd s\le C.
\]
This proves \eqref{eq:finite-global0}.

We next consider \(Z_\lambda\).  From
\eqref{eq:qlambda} and \eqref{eq:finite-qdelta},
\begin{equation}\label{eq:finite-weighted-outside}
 \int_{\{s\le0:Z_\lambda(s)\notin\mathcal U_\rho\}}
 e^{\lambda s}\,\dd s\le C.
\end{equation}
We use the same definition of the intervals \(I_j\): they are the connected components
of
\[
 \{s\le0:Z_\lambda(s)\in\mathcal U_{2\rho}\}
\]
satisfying $Z_\lambda(I_j)\cap\mathcal U_\rho\neq\emptyset$.  Again write \(a_j,b_j\) for their
endpoints.  If \(b_j<0\), define \(\sigma_j\) by the same formula as above,
with \(Z_0\) replaced by \(Z_\lambda\).  Then
\[
 b_j-\sigma_j\ge d_0:=\frac{\rho}{V},
 \qquad
 Z_\lambda(s)\notin\mathcal U_\rho
 \quad\text{for }s\in(\sigma_j,b_j].
\]
Consequently,
\[
 \int_{\sigma_j}^{b_j}e^{\lambda s}q(Z_\lambda(s))\,\dd s
 \ge \delta e^{\lambda b_j}\int_0^{b_j-\sigma_j}e^{-\lambda r}\,\dd r
 \ge \delta d_0e^{-\lambda_0d_0}e^{\lambda b_j}.
\]
The intervals \([\sigma_j,b_j]\) are disjoint.  By
\eqref{eq:qlambda} we conclude
\begin{equation}\label{eq:finite-exit-weights}
 \sum_{j:b_j<0}e^{\lambda b_j}\le C.
\end{equation}
There is at most one interval $I_j$ with \(b_j=0\).

For an interval $I_j$ with finite endpoints, using the same endpoint
approximation as for $Z_0$, the local estimate of
Lemma~\ref{lem:eq-local} or Lemma~\ref{lem:per-local}(i) gives
\[
 d_{T^*M}(Z_\lambda(s),\widetilde{\A})
 \le C\rho\bigl(e^{-\mu(s-a_j)}+e^{-\mu(b_j-s)}\bigr)+C\lambda
 \qquad \text{for every }s\in I_j.
\]
Since \(e^{\lambda s}\le e^{\lambda b_j}\) on \(I_j\),
\[
 \int_{I_j}e^{\lambda s}\rho
 \bigl(e^{-\mu(s-a_j)}+e^{-\mu(b_j-s)}\bigr)\,\dd s
 \le C e^{\lambda b_j}.
\]
The same bound follows from the left half-line estimate if \(a_j=-\infty\).
Therefore \eqref{eq:finite-exit-weights}, together with the one possible
interval with \(b_j=0\), controls all exponential terms.  The remaining
\(C\lambda\) terms satisfy
\[
 C\lambda\sum_j\int_{I_j}e^{\lambda s}\,\dd s
 \le C\lambda\int_{-\infty}^0e^{\lambda s}\,\dd s=C.
\]

Finally set
\[
 \mathcal R_\lambda:=(-\infty,0]\setminus\bigcup_j I_j.
\]
We have
\[
 \mathcal R_\lambda\subset
 \{s\le0:Z_\lambda(s)\notin\mathcal U_\rho\}.
\]
The distance to \(\widetilde{\A}\) is bounded on \(K\), so
\eqref{eq:finite-weighted-outside} gives
\[
 \int_{\mathcal R_\lambda}e^{\lambda s}
 d_{T^*M}(Z_\lambda(s),\widetilde{\A})\,\dd s\le C.
\]
Combining this with the estimates on the intervals $I_j$ proves
\eqref{eq:finite-global-lambda}.
\end{proof}

\subsection{Finite-time approximation of the Peierls barrier}

The Peierls barrier can be approximated by finite-time minimizers, but their durations will tend to infinity. 
The next lemma constructs a curve whose duration is $O(R)$ with action error $O(e^{-\kappa R})$.

\begin{lemma}\label{lem:peierls-compression}
There exist $C,\kappa>0$ such that, for every $x\in M$ and every
$R\ge1$, one can find an active component $\Gamma_i$, a point
$a\in\pi_M(\Gamma_i)$, a time $S\le CR+C$, and an absolutely continuous curve
$\eta_R:[0,S]\to M$ with $\eta_R(0)=a$ and $\eta_R(S)=x$, such that
\begin{equation}\label{eq:compression-defect}
 0\le
 \int_0^S\bigl(L(\eta_R,\dot\eta_R)+c\bigr)\,\dd t
 -\bigl(u_0(x)-u_0(a)\bigr)
 \le Ce^{-\kappa R}.
\end{equation}
\end{lemma}

In Subsection~\ref{subsec:finite-theorem-proof} we apply the lemma with $R=\frac1\kappa\log\frac1\lambda$.

\begin{proof}
Fix $x\in M$.  By Lemma~\ref{lem:active-representation}, choose an active
component $\Gamma_i$ and a point $a\in\pi_M(\Gamma_i)$ such that $h(a,x)=u_0(x)-u_0(a)$.
By the definition of the Peierls barrier there are $T_n\to\infty$ such that $h_{T_n}(a,x)\to h(a,x)$. Let $\gamma_n:[0,T_n]\to M$ be a Tonelli minimizer realizing
$h_{T_n}(a,x)$.  Define
\[
 p_n(t):=L_v(\gamma_n(t),\dot\gamma_n(t)),
 \qquad
 Z_n(t):=(\gamma_n(t),p_n(t))\in T^*M.
\]
By the Euler--Lagrange equation, $Z_n$ is an orbit of the conservative
Hamiltonian vector field $X_H$.  Set
\[
 D_n:=h_{T_n}(a,x)-\bigl(u_0(x)-u_0(a)\bigr).
\]
Domination of $u_0$ gives $D_n\ge0$, while the choice of $T_n$ gives
$D_n\to0$.  We next obtain a compact set containing these Hamiltonian trajectories,
uniformly for all sufficiently large $n$.  Since $T_n\to\infty$, we may
assume $T_n\ge1$.  The restriction of $\gamma_n$ to $[0,1]$ minimizes the
action between $\gamma_n(0)$ and $\gamma_n(1)$.  Let $\sigma_n:[0,1]\to M$
be a minimizing geodesic joining these two points, parametrized with constant
speed.  Since
\[
 d_M(\gamma_n(0),\gamma_n(1))\le \operatorname{diam}(M),
\]
we have $|\dot\sigma_n|\le\operatorname{diam}(M)$.  Compactness of $M$ then
gives a constant $C_0$, depending only on $H$ and $M$, such that
\[
 \int_0^1\bigl(L(\gamma_n,\dot\gamma_n)+c\bigr)\,\dd s
 \le
 \int_0^1\bigl(L(\sigma_n,\dot\sigma_n)+c\bigr)\,\dd s
 \le C_0.
\]
By superlinearity there are constants $A>0$ and $B$ such that
\[
 L(x,v)+c\ge A|v|-B
\]
for all $(x,v)\in TM$.  Hence,
\[
 \int_0^1|\dot\gamma_n(s)|\,\dd s\le C,
\]
so there exists $t_n\in[0,1]$ with
$|\dot\gamma_n(t_n)|\le C$.  Consequently
$p_n(t_n)=L_v(\gamma_n(t_n),\dot\gamma_n(t_n))$ belongs to a fixed compact
set.  Since $H$ is autonomous, $H(Z_n(t))$ is constant in $t$; its value is
therefore uniformly bounded.  Uniform coercivity of $H$ in $p$ now gives a
uniform bound for $p_n(t)$ for every $t\in[0,T_n]$.  Thus, there is a compact
set $K_1\subset T^*M$, depending only on $H$ and $M$, such that $Z_n(t)\in K_1$ for every $t\in[0,T_n]$
for all sufficiently large $n$.

Write the constant Hamiltonian energy in the form
\[
 H(Z_n(t))\equiv c+e_n.
\]
We claim that $e_n\to0$.  Indeed, since
$h_{T_n}(a,x)\to u_0(x)-u_0(a)$ and $u_0$ is bounded,
\[
 \int_0^{T_n}q(Z_n(t))\,\dd t
 =h_{T_n}(a,x)-w(x)+w(a)\le C_q
\]
for all sufficiently large $n$, where $C_q$ is independent of $x$ and of
the chosen $a$.

Suppose by contradiction that $e_n\not\to0$.  Passing to a subsequence,
there is $\varepsilon>0$ such that $|e_n|\ge\varepsilon$ for every $n$.
Since $Z_n([0,T_n])\subset K_1$ and $H(Z_n(t))\equiv c+e_n$,
the whole orbit is contained in the compact set
\[
 K_1\cap\{|H-c|\ge\varepsilon\}.
\]
This set is disjoint from $\widetilde{\A}=q^{-1}(0)$, because
$H=c$ on $\widetilde{\A}$.  Hence $q\ge\delta_\varepsilon>0$ there, and
therefore
\[
 C_q\ge\int_0^{T_n}q(Z_n(t))\,\dd t
 \ge\delta_\varepsilon T_n,
\]
which contradicts $T_n\to\infty$.  We conclude that $e_n\to0$, and
Lemma~\ref{lem:per-local}(ii) applies for all sufficiently large $n$.

We now modify the long minimizing curve.  Choose \(\rho>0\) so small that
the neighborhoods \(U_{2\rho}^i\), \(i=1,\dots,N\), are pairwise
disjoint and all local estimates of Section~\ref{sec:hyperbolic} hold with common constants. On \(K_1\setminus\mathcal U_\rho\), $q\ge\delta>0$.
Let \(V_1\) bound \(|\dot Z_n|\) on \(K_1\).

Let \(I_{n,j}\), \(1\le j\le J_n\), be the connected components of
\[
 \{t\in[0,T_n]:Z_n(t)\in\mathcal U_{2\rho}\}
\]
which reach \(\mathcal U_\rho\), i.e., $Z_n(I_{n,j})\cap\mathcal U_\rho\neq\emptyset$.  Set
\[
 a_{n,j}:=\inf I_{n,j},
 \qquad
 b_{n,j}:=\sup I_{n,j}.
\]
Then
\[
 Z_n(I_{n,j})\subset U_{2\rho}^{i(n,j)}
\]
for a unique component \(\Gamma_{i(n,j)}\).  Since $q\ge\delta$ on $K_1\setminus\mathcal U_\rho$ and
$|\dot Z_n|\le V_1$, the same crossing argument as in the proof of
Lemma~\ref{lem:finite-distance}, with at most two additional components meeting the endpoints $0$ and
$T_n$, gives
\begin{equation}\label{eq:number-visits}
   J_n\le J,
\end{equation}
where $J$ is independent of $n$ and $x$, and
\begin{equation}\label{eq:peierls-outside-time}
 \left|[0,T_n]\setminus\bigcup_{j=1}^{J_n}I_{n,j}\right|\le C.
\end{equation}

Let \(\mu>0\) be a common exponent in the local estimates and fix $0<\kappa<\mu$.
Fix \(R\ge1\).  Since \(D_n\to0\) and \(e_n\to0\),
we may choose \(n=n(R)\) so large that
\begin{equation}\label{eq:choose-peierls}
 D_n+|e_n|\le e^{-\kappa R},
 \qquad |e_n|\le e_0.
\end{equation}

By domination of $u_0$, for any finite family of pairwise disjoint
subintervals $[r_k,t_k]\subset[0,T_n]$,
\begin{equation}\label{eq:subinterval-defect}
 0\le
 \sum_k\left[
 \int_{r_k}^{t_k}
 \bigl(L(\gamma_n,\dot\gamma_n)+c\bigr)\,\dd\tau
 -u_0(\gamma_n(t_k))+u_0(\gamma_n(r_k))
 \right]
 \le D_n.
\end{equation}

Fix one interval \(I_{n,j}\).  If $b_{n,j}-a_{n,j}\le2R$,
we do not modify this part of \(\gamma_n\). Suppose now that
$b_{n,j}-a_{n,j}>2R$.
Define
\[
 s_{n,j}:=a_{n,j}+R,
 \qquad
 t_{n,j}:=b_{n,j}-R.
\]
Then
 $a_{n,j}<s_{n,j}<t_{n,j}<b_{n,j}$,
so both times belong to \(I_{n,j}\), and the whole interval between them is
still inside the same \(U_{2\rho}^{i(n,j)}\).

If
\(\Gamma_{i(n,j)}\) is an equilibrium, Lemma~\ref{lem:eq-local} with
\(\lambda=0\) gives
\[
\begin{aligned}
 d_{T^*M}(Z_n(s_{n,j}),\Gamma_{i(n,j)})
 &\le C\rho\left(e^{-\mu R}
 +e^{-\mu(b_{n,j}-a_{n,j}-R)}\right)
 \le Ce^{-\mu R},\\
 d_{T^*M}(Z_n(t_{n,j}),\Gamma_{i(n,j)})
 &\le C\rho\left(e^{-\mu(b_{n,j}-a_{n,j}-R)}
 +e^{-\mu R}\right)
 \le Ce^{-\mu R}.
\end{aligned}
\]
If \(\Gamma_{i(n,j)}\) is a periodic orbit, Lemma~\ref{lem:per-local} gives instead
\[
 d_{T^*M}(Z_n(s_{n,j}),\Gamma_{i(n,j)})
 +d_{T^*M}(Z_n(t_{n,j}),\Gamma_{i(n,j)})
 \le C\bigl(e^{-\mu R}+|e_n|\bigr).
\]
Since \(\kappa<\mu\) and \eqref{eq:choose-peierls} holds, in both cases
\begin{equation}\label{eq:close-to-component}
 d_{T^*M}(Z_n(s_{n,j}),\Gamma_{i(n,j)})
 +d_{T^*M}(Z_n(t_{n,j}),\Gamma_{i(n,j)})
 \le Ce^{-\kappa R}.
\end{equation}
Choose \(z_{n,j}^-,z_{n,j}^+\in\Gamma_{i(n,j)}\) realizing the two
distances in \eqref{eq:close-to-component}, and put
\[
 y_{n,j}^-:=\pi_M(z_{n,j}^-),
 \qquad
 y_{n,j}^+:=\pi_M(z_{n,j}^+).
\]
Then
\begin{equation}\label{eq:close-phases}
 d_M(\gamma_n(s_{n,j}),y_{n,j}^-)
 +d_M(\gamma_n(t_{n,j}),y_{n,j}^+)
 \le Ce^{-\kappa R}.
\end{equation}

We keep \(\gamma_n\) on \([a_{n,j},s_{n,j}]\) and
\([t_{n,j},b_{n,j}]\).  Between the two cut points, join
\(\gamma_n(s_{n,j})\) to \(y_{n,j}^-\) and \(y_{n,j}^+\) to
\(\gamma_n(t_{n,j})\) by minimizing geodesics parametrized with unit speed.
Their durations are $d_M(\gamma_n(s_{n,j}),y_{n,j}^-)$ and $d_M(y_{n,j}^+,\gamma_n(t_{n,j}))$,
and their sum is at most \(Ce^{-\kappa R}\) by
\eqref{eq:close-phases}.  Since the two geodesics have unit speed and \(M\) is compact, the
critical action along either geodesic is bounded by $Ce^{-\kappa R}$.  Together with the Lipschitz continuity and domination of \(u_0\),
this shows that the two connecting geodesics have total defect at most
\(Ce^{-\kappa R}\).

It remains to connect \(y_{n,j}^-\) to \(y_{n,j}^+\) along $\pi_M(\Gamma_{i(n,j)})$.  If \(\Gamma_{i(n,j)}\) is an equilibrium, these two
points coincide.  If it is a periodic orbit, move from \(z_{n,j}^-\) to
\(z_{n,j}^+\) in the forward Hamiltonian direction. This takes at most one
prime period.  Along $\pi_M(\Gamma_{i(n,j)})$,
\[
 \frac{\dd}{\dd t}u_0(x(t))
 =\langle p(t),\dot x(t)\rangle
 =L(x(t),\dot x(t))+c,
\]
which implies that the critical action of the segment of $\pi_M(\Gamma_{i(n,j)})$ is exactly
\(u_0(y_{n,j}^+)-u_0(y_{n,j}^-)\).  Together with the two short connecting
geodesics, this shows that the replacement of \(\gamma_n\) on
\([s_{n,j},t_{n,j}]\) has nonnegative defect bounded by \(Ce^{-\kappa R}\).

\begin{figure}[htbp]
\centering
\includegraphics[width=0.7\textwidth]{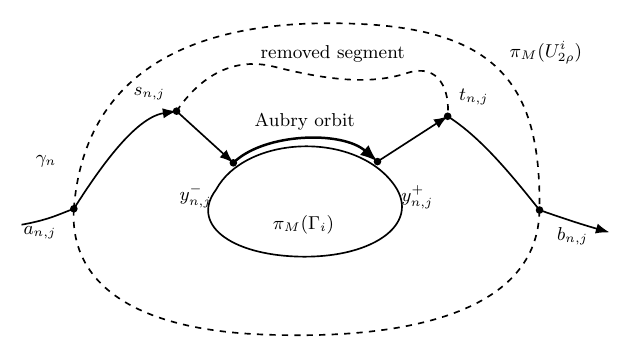}
\caption{Construction of $\eta_R$.}
\label{fig:peierls-compression}
\end{figure}

Perform this replacement on every interval $I_{n,j}$ of length greater than
\(2R\), and denote the resulting curve by \(\eta_R:[0,S]\to M\).  It still
joins \(a\) to \(x\).  By \eqref{eq:subinterval-defect}, the total defect of
the portions left unchanged is at most \(D_n\).  By
\eqref{eq:number-visits}, at most \(J\) replacements are made.  Therefore
\[
 0\le{}\int_0^S\bigl(L(\eta_R,\dot\eta_R)+c\bigr)\,\dd t
 -\bigl(u_0(x)-u_0(a)\bigr)
 \le{}D_n+CJe^{-\kappa R}\le Ce^{-\kappa R}.
\]
This proves \eqref{eq:compression-defect}.

We finally estimate the duration of the new curve.  By
\eqref{eq:peierls-outside-time}, the portions outside the intervals $I_{n,j}$
have total duration at most \(C\).  An interval $I_{n,j}$ of length at most
\(2R\) is unchanged and contributes at most \(2R\).  If an interval $I_{n,j}$
has length greater than \(2R\), the new curve keeps the first and last
pieces, each of duration \(R\), and replaces the middle by two geodesics of total duration at most \(Ce^{-\kappa R}\).  If $\Gamma_{i(n,j)}$ is a periodic orbit, the 
segment of it has duration at most one prime period. Since there are only finitely
many components of the Aubry set, these periods are uniformly bounded by a constant $P>0$.  Therefore
\[
 S\le C+J(2R+Ce^{-\kappa R}+P)\le CR+C.
\]
The proof is complete.
\end{proof}

\subsection{Proof of Theorem~\ref{thm:finite-main}}\label{subsec:finite-theorem-proof}

We first prove the lower bound in \eqref{eq:intro-finite-rate}.  In view of
\eqref{eq:lower}, it is enough to prove
\[
 \int_{-\infty}^0e^{\lambda s}u_0(\gamma_\lambda(s))\,\dd s\le C
\]
for every $u_\lambda$-calibrated curve ending at \(x\).  We now
construct a bounded function \(\Phi\) on phase space for which
\(X_H\Phi\) agrees with \(u_0-m_i\) on each $\Gamma_i$.  Away from the
components $\Gamma_i$, the error will be bounded by the distance to the Aubry set, and
Lemma~\ref{lem:finite-distance} will then give the required integral bound.

We first define \(\Phi\) near $\Gamma_i$.  If
\(\Gamma_i=\{(x_i,p_i)\}\) is an equilibrium, then $m_i=u_0(x_i)$, so \(u_0(x_i)-m_i=0\). Set
\(\Phi_i=0\) near $\Gamma_i$.  If \(\Gamma_i\) is a periodic orbit, choose a point \(z_i^0\in\Gamma_i\) and let
\(\varphi_H^t\) denote the Hamiltonian flow of \(X_H\).  Define
\[
 z_i(t):=\varphi_H^t(z_i^0)=(x_i(t),p_i(t)).
\]
Then
\[
 z_i(0)=z_i^0,
 \qquad
 \dot z_i(t)=X_H(z_i(t)),
 \qquad
 z_i(t+P_i)=z_i(t),
\]
where \(P_i\) is the prime period.  The invariant probability carried by
\(\Gamma_i\) is normalized time measure. We have
\[
 m_i=\frac1{P_i}\int_0^{P_i}u_0(x_i(t))\,\dd t.
\]
Define
\[
 F_i(t):=\int_0^t\bigl(u_0(x_i(s))-m_i\bigr)\,\dd s.
\]
The integrand has zero average over one period, so \(F_i\) is
\(P_i\)-periodic and bounded.

By the graph property of the Aubry set,
\[
 Du_0(x_i(t))=p_i(t).
\]
We have
\[
 \frac{\dd}{\dd t}u_0(x_i(t))
 =\langle p_i(t),\dot x_i(t)\rangle.
\]
The right-hand side is \(C^1\) in \(t\), so the restriction of \(u_0\) to
the projected orbit, and hence \(F_i\), is \(C^2\).  Define
\[
 \Phi_i(z_i(t)):=F_i(t)
 \qquad\text{on }\Gamma_i,
\]
and extend \(\Phi_i\) to a \(C^2\) function supported in a slightly larger
small neighborhood, without changing it near \(\Gamma_i\).  On the orbit,
\[
X_H\Phi_i(z_i(t))
 =\dot F_i(t)
 =u_0(x_i(t))-m_i.
\]
Since there are finitely many disjoint components $\Gamma_i$, their
neighborhoods can be chosen disjoint.  Using the functions just constructed,
and \(\Phi_i=0\) near equilibrium components, we obtain a bounded
\(C^2\) function \(\Phi\) on \(T^*M\) which satisfies
\[X_H\Phi=u_0\circ \pi_M-m_i\qquad \text{on }\Gamma_i,\ i=1,\dots,N.\]

We next extend the identity from each \(\Gamma_i\) to a nearby point with an
explicit error.  By the construction above, the function $u_0\circ\pi_M-m_i-X_H\Phi$ is Lipschitz on a fixed neighborhood of $\Gamma_i$ and vanishes on
$\Gamma_i$.  Hence
\begin{equation}\label{eq:Phi-distance}
 |u_0(\pi_Mz)-m_i-X_H\Phi(z)|
 \le C d_{T^*M}(z,\Gamma_i)
\end{equation}
for $z$ in this neighborhood.

Recall that the compact set \(K\) fixed in Subsection~\ref{subsec:variational} contains all
lifts of $u_0$- and $u_\lambda$-calibrated curves considered here.  Enlarging it
if necessary, we may also assume that it contains the finitely many fixed
neighborhoods used above.  Since the momentum coordinate is bounded on
\(K\),
\[
 X_\lambda(x,p)-X_H(x,p)=(0,-\lambda p),
 \qquad
 |X_\lambda-X_H|\le C\lambda
 \quad\text{on }K.
\]
As \(\Phi\in C^2\), this implies
\[
 |X_\lambda\Phi-X_H\Phi|\le C\lambda
 \quad\text{on }K.
\]
Combining this with \eqref{eq:Phi-distance},
\[
 |u_0(\pi_Mz)-m_i-X_\lambda\Phi(z)|
 \le C d_{T^*M}(z,\Gamma_i)+C\lambda
\]
near \(\Gamma_i\).  Let \(U\) be the union of these fixed neighborhoods.  Since
\(K\setminus U\) is compact and disjoint from \(\widetilde{\A}\), there is
\(d_*>0\) such that
\[
 d_{T^*M}(z,\widetilde{\A})\ge d_*
 \qquad\text{for every }z\in K\setminus U.
\]
Moreover \(u_0\circ\pi_M-X_\lambda\Phi\) is uniformly bounded on \(K\)
for \(0<\lambda\le\lambda_0\).  Using \(m_i\le0\) in $U$ and increasing \(C\) if necessary, we obtain the
global one-sided estimate
\begin{equation}\label{eq:finite-u0-upper}
 u_0(\pi_Mz)
 \le X_\lambda\Phi(z)
 +C d_{T^*M}(z,\widetilde{\A})+C\lambda
 \qquad \text{for }z\in K.
\end{equation}

Let $Z_\lambda(s)=(\gamma_\lambda(s),p_\lambda(s))$ be the lift of a $u_\lambda$-calibrated curve with
\(\gamma_\lambda(0)=x\). We have
\[
 X_\lambda\Phi(Z_\lambda(s))
 =\frac{\dd}{\dd s}\Phi(Z_\lambda(s)).
\]
Substituting \(z=Z_\lambda(s)\) in \eqref{eq:finite-u0-upper}, multiplying
by \(e^{\lambda s}\), and integrating gives
\[
\begin{aligned}
 \int_{-\infty}^0 e^{\lambda s}u_0(\gamma_\lambda(s))\,\dd s
 &\le
 \int_{-\infty}^0 e^{\lambda s}
 \frac{\dd}{\dd s}\Phi(Z_\lambda(s))\,\dd s\\
 &\quad+C\int_{-\infty}^0 e^{\lambda s}
 d_{T^*M}(Z_\lambda(s),\widetilde{\A})\,\dd s
 +C\lambda\int_{-\infty}^0e^{\lambda s}\,\dd s.
\end{aligned}
\]
Since \(\Phi\) is bounded,
\(e^{\lambda s}\Phi(Z_\lambda(s))\to0\) as \(s\to-\infty\).  We get
\[
 \int_{-\infty}^0 e^{\lambda s}
 \frac{\dd}{\dd s}\Phi(Z_\lambda(s))\,\dd s
 =\Phi(Z_\lambda(0))
 -\lambda\int_{-\infty}^0e^{\lambda s}\Phi(Z_\lambda(s))\,\dd s,
\]
which is bounded by \(2\|\Phi\|_\infty\).  The second term is
bounded by Lemma~\ref{lem:finite-distance}, while for the last term, $\lambda\int_{-\infty}^0e^{\lambda s}\,\dd s=1$.
Hence,
\[
 \int_{-\infty}^0e^{\lambda s}u_0(\gamma_\lambda(s))\,\dd s\le C.
\]
Substitution into \eqref{eq:lower} gives
\begin{equation}\label{eq:finite-lower}
 u_\lambda(x)-u_0(x)\ge-C\lambda.
\end{equation}

We next prove the upper bound.  We first record a linear estimate on the
active components.  Let $\Gamma_i$ be active and let
$a\in\A_i=\pi_M(\Gamma_i)$.  If $\Gamma_i$ is an equilibrium, the constant curve at $a$ is
$u_0$-calibrated. Then
$u_0(a)=m_i=0$, and \eqref{eq:upper} gives
\[
 u_\lambda(a)-u_0(a)\le0.
\]
If $\Gamma_i$ is periodic, choose $t_a$ such that $x_i(t_a)=a$.  Since
$m_i=0$, the function $F_i$ defined above satisfies $F_i'(t)=u_0(x_i(t))$ and is bounded and periodic.  Hence
\[
\begin{aligned}
 \int_{-\infty}^0 e^{\lambda s}u_0(x_i(t_a+s))\,\dd s
 &=
 F_i(t_a)
 -\lambda\int_{-\infty}^0
 e^{\lambda s}F_i(t_a+s)\,\dd s,
\end{aligned}
\]
and therefore its absolute value is bounded by $2\|F_i\|_\infty$.
Applying \eqref{eq:upper} to this calibrated orbit, and using the finiteness
of the components, we obtain
\begin{equation}\label{eq:active-point-upper}
 u_\lambda(a)-u_0(a)\le C\lambda
 \qquad
 \text{for every $a$ in an active component.}
\end{equation}

Fix now $x\in M$ and $R\ge1$.  By
Lemma~\ref{lem:peierls-compression}, there exist an active component
$\Gamma_i$, a point $a\in\A_i$, a time $S\le CR+C$, and an absolutely
continuous curve $\eta_R:[0,S]\to M$ joining $a$ to $x$ such that
\[
 D_R:=
 \int_0^S\bigl(L(\eta_R,\dot\eta_R)+c\bigr)\,\dd t
 -\bigl(u_0(x)-u_0(a)\bigr)
 \le Ce^{-\kappa R}.
\]
By domination of $u_0$, the absolutely continuous function
\[
 r_R(t):=
 \int_0^t\bigl(L(\eta_R,\dot\eta_R)+c\bigr)\,\dd s
 -\bigl(u_0(\eta_R(t))-u_0(a)\bigr)
\]
is nondecreasing.  We have $r_R(0)=0$, $r_R(S)=D_R$, and
$r_R'\ge0$ almost everywhere.  Integrating by parts gives
\[
\begin{aligned}
 &\int_0^S e^{-\lambda(S-t)}
 \bigl(L(\eta_R(t),\dot\eta_R(t))+c\bigr)\,\dd t\\
 &\qquad=
 u_0(x)-e^{-\lambda S}u_0(a)
 -\lambda\int_0^S e^{-\lambda(S-t)}u_0(\eta_R(t))\,\dd t
 +\int_0^S e^{-\lambda(S-t)}r_R'(t)\,\dd t\\
 &\qquad\le
 u_0(x)-e^{-\lambda S}u_0(a)
 +C\lambda S+D_R,
\end{aligned}
\]
where in the last inequality we used $0<e^{-\lambda(S-t)}\le1$ and $r_R'\ge0$. By the dynamic programming principle, we have
\[
 u_\lambda(x)
 \le
 e^{-\lambda S}u_\lambda(a)
 +\int_0^S e^{-\lambda(S-t)}
 \bigl(L(\eta_R(t),\dot\eta_R(t))+c\bigr)\,\dd t.
\]
Combining the last two inequalities with
\eqref{eq:active-point-upper}, $S\le CR+C$, and
$D_R\le Ce^{-\kappa R}$, we obtain
\begin{equation}\label{eq:finite-upper-R}
 u_\lambda(x)-u_0(x)
 \le C\lambda R+Ce^{-\kappa R}+C\lambda.
\end{equation}
After decreasing \(\lambda_0\) if necessary, assume
\(0<\lambda<\min\{e^{-\kappa},e^{-1}\}\) and choose
\[
 R=\frac1\kappa\log\frac1\lambda.
\]
Then \(e^{-\kappa R}=\lambda\), and \eqref{eq:finite-upper-R} gives
\[
 u_\lambda(x)-u_0(x)
 \le C\lambda\left(1+\log\frac1\lambda\right)
 \le C\lambda|\log\lambda|.
\]
The constants are uniform in \(x\).  Together with \eqref{eq:finite-lower},
this proves \eqref{eq:intro-finite-rate}.

\smallskip

Finally assume that $m_i=0$ for every $i$. Then
\eqref{eq:Phi-distance} gives
\[
 \bigl|u_0(\pi_Mz)-X_H\Phi(z)\bigr|
 \le C d_{T^*M}(z,\widetilde{\A}),
\]
and
\[
 \bigl|u_0(\pi_Mz)-X_\lambda\Phi(z)\bigr|
 \le C d_{T^*M}(z,\widetilde{\A})+C\lambda.
\]
As before, these inequalities are first obtained in fixed neighborhoods of
$\Gamma_i$ and then extended to all of $K$ by enlarging $C$.

Fix $x\in M$, and let $\gamma_0$ and $\gamma_\lambda$ be 
$u_0$- and $u_\lambda$-calibrated curves ending at $x$, with lifts
$Z_0$ and $Z_\lambda$. We have
\[
 X_H\Phi(Z_0(s))=\frac{\dd}{\dd s}\Phi(Z_0(s)),
 \qquad
 X_\lambda\Phi(Z_\lambda(s))
 =\frac{\dd}{\dd s}\Phi(Z_\lambda(s)).
\]
Using \eqref{eq:finite-global0},
\eqref{eq:finite-global-lambda}, and integration by parts, we obtain
\[
 \left|
 \int_{-\infty}^0e^{\lambda s}u_0(\gamma_0(s))\,\dd s
 \right|
 +
 \left|
 \int_{-\infty}^0e^{\lambda s}u_0(\gamma_\lambda(s))\,\dd s
 \right|
 \le C.
\]
Here, 
$e^{\lambda s}\le1$, so the unweighted estimate
\eqref{eq:finite-global0} is sufficient.  Thus the two integrals appearing
in \eqref{eq:upper} and \eqref{eq:lower} are uniformly bounded.  We conclude
\[
 |u_\lambda(x)-u_0(x)|\le C\lambda
\]
uniformly in $x$.  This proves \eqref{eq:intro-active-rate} and completes the
proof of Theorem~\ref{thm:finite-main}.

\begin{proof}[Proof of Corollary~\ref{cor:single-intro}]
If $N=1$, Lemma~\ref{lem:active-representation} implies that the unique
component is active.  The second part of Theorem~\ref{thm:finite-main}
therefore gives
$\|u_\lambda-u_0\|_\infty\le C\lambda$.
\end{proof}

\section{One-well examples: sharp linear rate and degenerate cases}\label{sec:one-well}

We give two types of one-dimensional examples.  The first family shows that
the linear rate is optimal in the hyperbolic case and gives slower algebraic
lower bounds when the equilibrium in the Aubry set is degenerate of finite order.  The second
example shows that, when the degeneracy is of infinite order, the convergence
can be arbitrarily slow.

\subsection{Finite-order examples}

Let $M=\T^1$ and fix an integer $r\ge1$.  Choose $V\in C^\infty(\T)$ such
that $V\le0$, $V^{-1}(0)=\{0\}$, and, in a local coordinate centered at $0$,
\[
V(x)=-x^{2r}
\]
for $|x|$ small.  We also choose the periodic completion so that $V$ has a
single other critical point, which is a nondegenerate minimum.  Set
\[
a_*:=\int_{\T}\sqrt{-2V(x)}\,\dd x
\]
and fix $a$ with $0<a<a_*$.  Consider
\[
H_a(x,p)=\frac12(p+a)^2+V(x).
\]
Its Lagrangian is
\[
L_a(x,v)=\frac12v^2-av-V(x).
\]
Define a periodic
$C^1$ function $w$ by
\[
w'(x)=\frac{a}{a_*}\sqrt{-2V(x)}-a,
\]
where the right-hand side has zero average.  Then
\[
H_a(x,w'(x))
=\left(1-\frac{a^2}{a_*^2}\right)V(x)\le0,
\]
with strict inequality for $x\neq0$.  The subsolution above gives $c(H_a)\le0$.  On the other hand,
$\max_x\min_p H_a(x,p)=\max V=0$, so $c(H_a)\ge0$ and therefore
$c(H_a)=0$.  The invariant measure supported at $(x,v)=(0,0)$ has zero
critical action and is minimizing.  Hence $0$ belongs to the projected Mather
set, and therefore to $\A(H_a)$.
Since the critical subsolution $w$ is strict at every $x\ne0$, $\A(H_a)\subset\{0\}$. We have $\A(H_a)=\{0\}$.  Finally the lifted Aubry set lies in the
critical energy level, and
\[
 0=H_a(0,p)=\frac12(p+a)^2
\]
forces $p=-a$.  Hence, $\widetilde{\A}(H_a)=\{(0,-a)\}$.

Let $u_\lambda$ be the solution of
\begin{equation}\label{eq:ex-disc}
\lambda u_\lambda+\frac12(u_\lambda'+a)^2+V(x)=0\qquad\text{in }\T.
\end{equation}
\begin{proposition}\label{prop:finite-order-rate}
There exists $C_r>0$ such that
\[
  \|u_\lambda-u_0\|_\infty
  \ge C_r\lambda^{1/(2r-1)}
\]
for all sufficiently small $\lambda>0$. If $r=1$, then
\[
  \|u_\lambda-u_0\|_\infty\asymp\lambda.
\]
If $r\ge2$, the estimate
$\|u_\lambda-u_0\|_\infty=O(\lambda)$ fails.
\end{proposition}

To prove the proposition, we first identify the discounted Mather set
$\widetilde{\M}_\lambda\subset T^*\T$ associated with the exact conformally
symplectic flow of \eqref{eq:ex-disc}.  It is nonempty, compact and invariant,
and it is contained in the graph of $Du_\lambda$ over its projection; see
\cite{MaroSorrentino}.  Since it is the closure of the union of supports of
invariant minimizing measures, Poincar\'e recurrence also implies that
recurrent points are dense in $\widetilde{\M}_\lambda$.

\begin{lemma}\label{lem:example-mather}
For all sufficiently small $\lambda>0$,
\[
\widetilde{\M}_\lambda=\{(x_\lambda,-a)\},
\qquad
x_\lambda=-\left(\frac{a\lambda}{2r}\right)^{1/(2r-1)}.
\]
\end{lemma}

\begin{proof}
The discounted Hamiltonian equations are
\[
\dot x=p+a,\qquad \dot p=-V'(x)-\lambda p.
\]
An equilibrium therefore satisfies $p=-a$ and $V'(x)=\lambda a$.  For
small $\lambda$, there are exactly two such equilibria.  Indeed, away from
fixed small neighborhoods of the two critical points of $V$, the quantity
$|V'|$ is bounded below by a positive constant.  Near $0$ the prescribed
local form makes $V'$ strictly monotone on each side, while near the other
critical point nondegeneracy gives the same conclusion.  Thus one solution
$x_\lambda$ lies near the degenerate maximum $0$ and one solution
$y_\lambda$ lies near the nondegenerate minimum of $V$.  Since
$V'(x)=-2r x^{2r-1}$ near $0$,
\[
x_\lambda=-\left(\frac{a\lambda}{2r}\right)^{1/(2r-1)}.
\]

We first exclude $(y_\lambda,-a)$.  If this equilibrium belonged to the
Mather set, the corresponding constant curve would be minimizing for the
discounted action.  Since $L_a(x,-v)=\frac12v^2+av-V(x)$, reversing time in the variational
representation of \eqref{eq:ex-disc} gives the functional
\[
I(\gamma)=\int_0^\infty e^{-\lambda t}
\left(\frac12|\dot\gamma|^2+a\dot\gamma-V(\gamma)\right)\dd t,
\qquad \gamma(0)=y_\lambda.
\]
For a variation $\gamma_\varepsilon=y_\lambda+\varepsilon\eta$ with
$\eta(0)=0$ and compact support, the first variation vanishes because
$V'(y_\lambda)=\lambda a$, while
\[
\delta^2I(\eta)=\int_0^\infty e^{-\lambda t}
\left(|\dot\eta|^2-V''(y_\lambda)\eta^2\right)\dd t.
\]
Since $y_\lambda$ converges to a nondegenerate minimum of $V$, there is
$\kappa>0$ such that $V''(y_\lambda)\ge\kappa$ for small $\lambda$.
Writing $\eta(t)=e^{\lambda t/2}\phi(t)$ with
$\phi\in C_c^\infty((0,\infty))$ gives
\[
\delta^2I(\eta)=\int_0^\infty
\left(|\dot\phi|^2+
\left(\frac{\lambda^2}{4}-V''(y_\lambda)\right)|\phi|^2\right)\dd t.
\]
For small $\lambda$, this is negative for a smooth function supported on a
sufficiently long interval, a contradiction.  Thus, 
$(y_\lambda,-a)\notin\widetilde{\M}_\lambda$.

It remains to exclude rotational periodic minimizing orbits.  Suppose that
for a sequence $\lambda_n\to0$ the set
$\widetilde{\M}_{\lambda_n}$ contains such an orbit.  Along a rotational orbit in the 
Mather set, $u_{\lambda_n}$ is differentiable.  Since
$\dot x=Du_{\lambda_n}+a$ has a fixed sign on such an orbit, the same branch
is used throughout one complete turn.  By \eqref{eq:ex-disc},
\[
Du_{\lambda_n}
=-a\pm\sqrt{-2\bigl(V+\lambda_nu_{\lambda_n}\bigr)}.
\]
By periodicity of $u_{\lambda_n}$,
\[
0=-a\pm\int_{\T}
\sqrt{-2\bigl(V(x)+\lambda_nu_{\lambda_n}(x)\bigr)}\,\dd x.
\]
The family $u_\lambda$ is uniformly bounded.  Hence the integrand converges
uniformly to $\sqrt{-2V}$, and passing to the limit yields $a=\pm a_*$,
contrary to $0<a<a_*$.

Finally, recurrent points are dense in $\widetilde{\M}_\lambda$.  In one
dimension a recurrent orbit in $\widetilde{\M}_\lambda$ is either a fixed
point or a rotational periodic orbit.  Indeed, a nonfixed recurrent orbit on
a graph over $\T$ cannot have a turning point, since this would force the
same base point to be visited with two different momenta.  Its projected
velocity therefore has a fixed sign. Recurrence rules out convergence to a
fixed point, so the projection winds around the circle.  After one turn the
orbit returns to the same point of the graph of $Du_\lambda$ and is therefore
periodic.  We have
excluded the latter and also the only other fixed point
$(y_\lambda,-a)$.  Therefore
$\widetilde{\M}_\lambda\subset\{(x_\lambda,-a)\}$.  Since the Mather set is
nonempty, equality follows.
\end{proof}

\begin{proof}[Proof of Proposition~\ref{prop:finite-order-rate}]
By Lemma~\ref{lem:example-mather},
$(x_\lambda,-a)\in\widetilde{\M}_\lambda$.  Since
$\widetilde{\M}_\lambda$ is contained in the graph of $Du_\lambda$,
$u_\lambda$ is differentiable at $x_\lambda$ and
$u_\lambda'(x_\lambda)=-a$.  Evaluating \eqref{eq:ex-disc} at $x_\lambda$
gives
\[
u_\lambda(x_\lambda)
=-\frac{V(x_\lambda)}{\lambda}
=\frac{x_\lambda^{2r}}{\lambda}
=-\frac{a}{2r}x_\lambda
=\frac{a}{2r}|x_\lambda|.
\]

The limiting Aubry set consists only of the point $0$, so the only projected
Mather measure is $\delta_0$.  The representation formula of \cite{DFIZ}
therefore gives $u_0(x)=h(0,x)$ and in particular $u_0(0)=0$.  Since $u_0$ solves the critical equation, for almost every $x$
sufficiently close to $0$,
\[
u_0'(x)=-a\pm\sqrt{-2V(x)}=-a\pm\sqrt2\,|x|^r.
\]
Therefore
\[
\big|u_0(x_\lambda)-a|x_\lambda|\big|
\le \frac{\sqrt2}{r+1}|x_\lambda|^{r+1}.
\]
Combining this with the value of $u_\lambda(x_\lambda)$ gives, for all
sufficiently small $\lambda$,
\[
u_0(x_\lambda)-u_\lambda(x_\lambda)
\ge a\left(1-\frac1{2r}\right)|x_\lambda|-\frac{\sqrt2}{r+1}|x_\lambda|^{r+1}
\ge C_r|x_\lambda|.
\]
Hence,
\[
\|u_\lambda-u_0\|_\infty
\ge C_r\lambda^{1/(2r-1)}.
\]
If $r=1$, then $V''(0)=-2$ and the equilibrium in the Aubry set is hyperbolic.  In this
case the estimate above reads $\|u_\lambda-u_0\|_\infty\ge C_1\lambda$,
while Corollary~\ref{cor:single-intro} gives the upper bound.  We conclude that $\|u_\lambda-u_0\|_\infty\asymp\lambda$,
and the linear rate is optimal even for a one-dimensional mechanical
Hamiltonian with a single hyperbolic equilibrium in the Aubry set.

If $r\ge2$, then $V''(0)=0$ and $1/(2r-1)<1$, so the $O(\lambda)$ estimate
fails.  Moreover, the exponents $1/(2r-1)$ tend to zero as $r\to\infty$.
Thus, even within this elementary one-dimensional class, there is no positive
algebraic convergence exponent valid for every Hamiltonian once the
hyperbolicity assumption is removed.
\end{proof}

\subsection{Arbitrarily slow convergence}\label{subsec:arbitrarily-slow}

The following construction was provided by Hung V. Tran.
\begin{proposition}\label{prop:arbitrarily-slow}
Let $\omega:[0,\infty)\to[0,\infty)$ be nondecreasing, with $\omega(0)=0$ and
$\omega(r)\to0$ as $r\downarrow0$.  There exists $V\in C^\infty(\T)$ such that
\[
  V\le0,\qquad V^{-1}(0)=\{0\},\qquad
  V^{(m)}(0)=0\quad\text{for every }m\ge1,
\]
and, for $H(x,p)=\frac12p^2+V(x)$, the corresponding discounted solutions satisfy
\begin{equation}\label{eq:arbitrary-slow-goal}
  \|u_\lambda-u_0\|_\infty\ge\omega(\lambda)
\end{equation}
for every sufficiently small $\lambda>0$.
\end{proposition}

Set $\alpha_k:=2^{-k^2-4}$ for $k\ge1$.  Since $\omega(r)\to0$ as $r\downarrow0$, we may
choose a strictly decreasing sequence $\lambda_k\downarrow0$ such that
\begin{equation}\label{eq:arbitrary-lambda}
  \lambda_{k+1}<\frac12\lambda_k,
  \qquad
  \omega(\lambda_k)\le(1-e^{-1})\alpha_k
  \qquad\text{for every }k\ge1.
\end{equation}

Fix $0<R<1/2$ and set $r_k:=2^{-k}R$ for $k\ge 1$.

The construction uses two different parts of $\sqrt{-2V}$ on each interval
$(r_{k+1},r_k)$.  We place smooth bumps away from a closed interval $D_k$ so
that the limiting solution will satisfy $u_0(r_k)\ge\alpha_k$.  We then add a
smooth function $\eta$ which is positive on $(0,R]$ but very small on each
$D_k$.  Its positivity ensures that $V(x)<0$ for $x\ne0$, which implies
$\A=\{0\}$. Its smallness on $D_k$ makes the calibrated curve spend a
long time crossing that interval.  The discounted representation formula and
an integration by parts then turn these two estimates into the lower bound
\eqref{eq:arbitrary-slow-goal}.

For every $k\ge1$, choose a closed interval
$D_k\Subset(r_{k+1},r_k)$ of length $(r_k-r_{k+1})/4$.  We also choose
$\psi_k\in C_c^\infty((r_{k+1},r_k)\setminus D_k)$ such that
\[
  \psi_k\ge0,\qquad
  \int_{r_{k+1}}^{r_k}\psi_k(x)\,\dd x=\alpha_k-\alpha_{k+1},\qquad
  \|\psi_k^{(m)}\|_{L^\infty}
  \le C_m(\alpha_k-\alpha_{k+1})r_k^{-m-1}
\]
for every $m\ge0$.  This can be done by rescaling a fixed smooth bump.  Since
$\alpha_k\le2^{-k^2-4}$ and $r_k\asymp2^{-k}$,
\[
  \|\psi_k^{(m)}\|_{L^\infty}
  \le C_m2^{-k^2+k(m+1)}\longrightarrow0
  \qquad\text{as }k\to\infty.
\]
The supports of the $\psi_k$ are disjoint and converge to $0$, so
$\sum_{k=1}^\infty\psi_k$ extends smoothly to $0$, with all derivatives equal to zero there.

We next construct $\eta\in C^\infty([0,R])$ which is positive on $(0,R]$, has all
its derivatives equal to zero at $0$, and is sufficiently small on each $D_k$.  Choose a
locally finite smooth partition of unity $\{\theta_j\}_{j\ge1}$ on $(0,R]$ such that, with
$r_0:=2R$,
\[
  0\le\theta_j\le1,\qquad
  \operatorname{supp}\theta_j\subset(r_{j+2},r_{j-1}),\qquad
  \|\theta_j^{(m)}\|_{L^\infty}\le C_mr_j^{-m}
\]
for every $m\ge0$.  We may arrange that at most four of the functions $\theta_j$ are
nonzero at any point and that
\[
  \operatorname{supp}\theta_j\cap D_k\ne\varnothing
  \quad\Longrightarrow\quad |j-k|\le3.
\]
Choose $\varepsilon_j>0$ so small that
\[
  \varepsilon_j\le r_j^j,
  \qquad
  \varepsilon_j\le\frac14|D_k|\lambda_{k+2}
  \quad\text{whenever }\operatorname{supp}\theta_j\cap D_k\ne\varnothing,
\]
and define
\[
  \eta(x):=\sum_{j=1}^\infty\varepsilon_j\theta_j(x)
  \qquad\text{for }x\in(0,R],
  \qquad \eta(0):=0.
\]
Then $\eta(x)>0$ for $x\in(0,R]$.  If $x\in D_k$, at most four terms occur, so
\begin{equation}\label{eq:arbitrary-eta}
  \eta(x)\le |D_k|\lambda_{k+2}.
\end{equation}
Moreover, for every fixed $m\ge0$ and $x\in(r_{k+1},r_k)$,
\[
  |\eta^{(m)}(x)|
  \le C_m\max_{|j-k|\le3}\varepsilon_jr_j^{-m}
  \le C_m\max_{|j-k|\le3}r_j^{j-m}\longrightarrow0
  \qquad\text{as }k\to\infty.
\]
Hence $\eta\in C^\infty([0,R])$ and $\eta^{(m)}(0)=0$ for every $m\ge0$.
Since $D_k\subset (0,r_1)$ for every $k\ge 1$, we may modify $\eta$ on
$[r_1,R]$, without affecting any of the preceding estimates, so that it is
positive and constant in a neighborhood of $R$.

On $[0,R]$, define
\[
  V(x):=-\frac12\left(\eta(x)+\sum_{k=1}^\infty\psi_k(x)\right)^2.
\]
Then $V\in C^\infty([0,R])$, $V(0)=0$, $V(x)<0$ for $x\in(0,R]$, and all derivatives
of $V$ vanish at $0$.  Since $\sum_{k=1}^\infty\psi_k$ vanishes near $R$ and
$\eta$ is constant there, $V$ is constant in a neighborhood of $R$.
Extend $V$ smoothly to $[R,1]$ so that
\[
  V(1)=0,\qquad V(x)<0\quad\text{for }R\le x<1,
  \qquad V^{(m)}(1)=0\quad\text{for every }m\ge1,
\]
and
\begin{equation}\label{eq:arbitrary-long-arc}
  \int_R^1\sqrt{-2V(x)}\,\dd x
  >\int_0^R\sqrt{-2V(x)}\,\dd x.
\end{equation}
This gives a smooth function on $\T$ satisfying the properties above.

For $H(x,p)=\frac12p^2+V(x)$, the same argument as in the preceding
subsection gives $c(H)=0$ and $\A=\{0\}$. The Lagrangian is $L(x,v)=\frac12v^2-V(x)$.  Since $L\ge0$,
the discounted representation formula gives $u_\lambda\ge0$. Using the constant curve at $0$, we conclude $u_\lambda(0)=0$.
The selected critical solution is
\[
  u_0(x)=\min\left\{
    \int_0^x\sqrt{-2V(y)}\,\dd y,
    \int_x^1\sqrt{-2V(y)}\,\dd y
  \right\}.
\]
By \eqref{eq:arbitrary-long-arc},
\begin{equation}\label{eq:arbitrary-u0}
  u_0(x)=\int_0^x\sqrt{-2V(y)}\,\dd y
  \qquad\text{for }0\le x\le R.
\end{equation}

Let $\gamma:(-\infty,0]\to(0,R]$ be the solution of
\[
  \dot\gamma(s)=\sqrt{-2V(\gamma(s))},
  \qquad \gamma(0)=R.
\]
For $k\ge2$, the interval $D_{k-1}$ lies in $(r_k,r_{k-1})$, and every $\psi_j$ vanishes
on $D_{k-1}$.  Hence,
\[
  \sqrt{-2V(x)}=\eta(x)
  \qquad\text{for }x\in D_{k-1}.
\]
By \eqref{eq:arbitrary-eta}, the time needed to travel from $r_k$ to $R$ satisfies
\begin{equation}\label{eq:arbitrary-travel}
  \int_{r_k}^R\frac{\dd x}{\sqrt{-2V(x)}}
  \ge\int_{D_{k-1}}\frac{\dd x}{\eta(x)}
  \ge\frac{|D_{k-1}|}{\sup_{D_{k-1}}\eta}
  \ge\frac1{\lambda_{k+1}}.
\end{equation}
Since $\lambda_k\downarrow0$, it follows that
$\int_0^R(-2V(x))^{-1/2}\,\dd x=\infty$.  Thus $\gamma$ is defined on
$(-\infty,0]$ and $\gamma(s)\to0$ as $s\to-\infty$.

Along $\gamma$, \eqref{eq:arbitrary-u0} gives
\[
  L(\gamma(s),\dot\gamma(s))
  =-2V(\gamma(s))
  =\frac{\dd}{\dd s}u_0(\gamma(s)).
\]
Using $\gamma$ as a competitor in the discounted representation formula and integrating by
parts, we obtain
\[
  u_\lambda(R)
  \le \int_{-\infty}^0 e^{\lambda s}
     \frac{\dd}{\dd s}u_0(\gamma(s))\,\dd s
  =u_0(R)-\lambda\int_{-\infty}^0
     e^{\lambda s}u_0(\gamma(s))\,\dd s.
\]
Since both $\gamma$ and $u_0$ are increasing on the relevant intervals,
$s\mapsto u_0(\gamma(s))$ is nondecreasing.  We have
\begin{align}
  u_0(R)-u_\lambda(R)
  &\ge \lambda\int_{-\infty}^0 e^{\lambda s}u_0(\gamma(s))\,\dd s \notag\\
  &\ge \lambda\int_{-1/\lambda}^0 e^{\lambda s}u_0(\gamma(s))\,\dd s  \ge (1-e^{-1})u_0(\gamma(-1/\lambda)).
  \label{eq:arbitrary-error}
\end{align}

For every $k\ge1$, the supports of $\psi_j$, $j\ge k$, are contained in $(0,r_k)$.  Hence,
\[
  u_0(r_k)
  =\int_0^{r_k}\sqrt{-2V(x)}\,\dd x \ge\sum_{j=k}^\infty\int_{r_{j+1}}^{r_j}\psi_j(x)\,\dd x
   =\sum_{j=k}^\infty(\alpha_j-\alpha_{j+1})
   =\alpha_k.
\]

Fix $k\ge2$ and let $\lambda\in[\lambda_{k+1},\lambda_k]$.  Then
$1/\lambda\le1/\lambda_{k+1}$, so \eqref{eq:arbitrary-travel} implies
$\gamma(-1/\lambda)\ge r_k$.
Since $u_0$ is increasing on $[0,R]$,
\[
  u_0(\gamma(-1/\lambda))\ge u_0(r_k)\ge\alpha_k.
\]
Therefore, by
\eqref{eq:arbitrary-error}, \eqref{eq:arbitrary-lambda}, and the monotonicity of $\omega$,
\[
  u_0(R)-u_\lambda(R)
  \ge(1-e^{-1})\alpha_k
  \ge\omega(\lambda_k)
  \ge\omega(\lambda).
\]
The intervals $[\lambda_{k+1},\lambda_k]$, $k\ge2$, cover $(0,\lambda_2]$.  We conclude that
\[
  \|u_\lambda-u_0\|_\infty
  \ge u_0(R)-u_\lambda(R)
  \ge \omega(\lambda)
\]
for every $0<\lambda\le\lambda_2$, which proves
\eqref{eq:arbitrary-slow-goal}.

\section{Sharpness of the logarithmic rate: a two-well model}\label{sec:two-well}

We now show that the logarithmic loss in Theorem~\ref{thm:finite-main} is
optimal.  The example is an asymmetric two-well mechanical model; see
\cite[Section~5.3]{ZavBook} for related discussions.  The sharp
lower bound comes from the local branch structure of the discounted solution
near the equilibrium $x=1/2$, where the selected critical solution is
strictly negative.

Choose
$V\in C^\infty(\T)$ such that
\[
 V\le0,\qquad
 V^{-1}(0)=\left\{0,\frac12\right\},\qquad
 V''(0)<0,\quad V''\left(\frac12\right)<0.
\]
Set
\[
 Q(x):=\sqrt{-2V(x)},\qquad
 A:=\int_0^{1/2}Q(s)\,\dd s,\qquad
 B:=\int_{1/2}^1Q(s)\,\dd s,
\]
and choose the two sides asymmetrically so that
$A<B$.
Set
\[
 a:=\frac{3A+B}{2}.
\]
Then $2A<a<A+B$.  Consider
\begin{equation}\label{eq:twowell-H}
 H_a(x,p)=\frac12(p+a)^2+V(x).
\end{equation}

\begin{lemma}\label{lem:twowell-geometry}
For the Hamiltonian \eqref{eq:twowell-H}, the critical value is $0$, and
\[
 \widetilde{\A}(H_a)
 =\left\{(0,-a),\left(\frac12,-a\right)\right\}.
\]
The two points are distinct static classes and are hyperbolic equilibria.
Moreover, with
\begin{equation}\label{eq:twowell-Delta}
\Delta:=\frac a2-A=\frac{B-A}{4}>0,
\end{equation}
the selected critical solution satisfies
\begin{equation}\label{eq:twowell-u0-values}
u_0(0)=0,\qquad
u_0\left(\frac12\right)=-\Delta,
\end{equation}
and
\begin{equation}\label{eq:twowell-u0-branch}
u_0(x)=\int_0^x\bigl(Q(s)-a\bigr)\,\dd s,
\qquad
u_0'(x)+a=Q(x),
\qquad 0\le x\le\frac12.
\end{equation}
\end{lemma}

\begin{proof}
In the notation of \cite[Section~5.3]{ZavBook}, we have
$X=1/2$, $f^+=Q$, and
\[
 \alpha=\frac1X\int_0^X f^+(s)\,\dd s=2A.
\]
Since $2A<a<A+B=\int_\T Q$, the results there give
$c(H_a)=0$, $\A(H_a)=\{0,1/2\}$, and
\[
 u_0(x)=\int_0^x(Q(s)-a)\,\dd s,
 \qquad 0\le x\le\frac12.
\]
This yields \eqref{eq:twowell-u0-values} and
\eqref{eq:twowell-u0-branch}.  Since the lifted Aubry set lies in
$H_a^{-1}(0)$, the momentum above $0$ and $1/2$ is $-a$.

It remains to separate the two static classes.  Set
$\vartheta:=a/(A+B)\in(0,1)$ and let $w$ be the periodic $C^1$
function with $w'=\vartheta Q-a$.  Then
\[
 H_a(x,w'(x))
 =-\frac{1-\vartheta^2}{2}Q(x)^2\le0,
\]
so $w$ is a critical subsolution.  On the other hand,
\[
 u_0\left(\frac12\right)-u_0(0)=A-\frac a2,
 \qquad
 w\left(\frac12\right)-w(0)=\vartheta A-\frac a2.
\]
These two increments are different.  Since all critical subsolutions have
the same increment between two points in the same static class, $0$ and
$1/2$ belong to distinct static classes.

Finally, the linearization at either equilibrium is
\[
 \begin{pmatrix}
 0&1\\
 -V''(x)&0
 \end{pmatrix},
\]
whose eigenvalues are real and nonzero because $V''(0)$ and
$V''(1/2)$ are negative.
\end{proof}

Since the Aubry set consists of two equilibria, their projected Mather measures are
$\delta_0$ and $\delta_{1/2}$.  We have
\[
 m_1=u_0(0)=0,
 \qquad
 m_2=u_0(1/2)=-\Delta<0.
\]
Thus the component at $0$ is active, whereas the component at $1/2$ is not.

The discounted equation in the present example is
\[
 \lambda u_\lambda
 +\frac12(u_\lambda'+a)^2+V(x)=0
 \qquad\text{in }\T.
\]
The source of the logarithm is the square-root splitting of the two branches
of this equation near the well $x=1/2$.  Near this point the critical solution
satisfies $u_0'+a=Q$, whereas the discounted equation has two possible roots
for $u_\lambda'+a$.  As $x$ increases toward $1/2$, the discounted solution
may in principle change from the positive root to the negative one.
The following elementary one-dimensional observation shows that such a change
can occur at most once and only in this direction.

\begin{lemma}\label{lem:twowell-branch}
Let $J\subset\T$ be a closed arc on which
$u_\lambda<0$, and identify $J$ with a compact interval in $\R$.  At every
differentiability point of $u_\lambda$ in $J$,
\[
 u_\lambda'(x)+a
 =\pm\sqrt{Q(x)^2-2\lambda u_\lambda(x)}.
\]
Moreover, the negative root cannot occur to the left of the positive root.
Thus there exists $\sigma\in J$ such that, up to a set of measure
zero,
\[
 u_\lambda'(x)+a=\sqrt{Q(x)^2-2\lambda u_\lambda(x)}
 \quad\text{for }x<\sigma,
\]
and
\[
 u_\lambda'(x)+a=-\sqrt{Q(x)^2-2\lambda u_\lambda(x)}
 \quad\text{for }x>\sigma.
\]
Here $\sigma=\inf J$ corresponds to the negative root almost everywhere on
$J$, and $\sigma=\sup J$ to the positive root almost everywhere on $J$.
\end{lemma}

\begin{proof}
The solution $u_\lambda$ is Lipschitz and hence differentiable almost
everywhere.  At a differentiability point, the discounted equation gives the
two roots above. Since $u_\lambda<0$, the radicand is strictly positive.

Set
$v(x):=u_\lambda(x)+ax$.
The function $v$ has no local minimum in the interior of $J$.  Indeed, if
$y\in J$ were a local minimum, then
$\phi(x):=u_\lambda(y)-a(x-y)$
would touch $u_\lambda$ from below at $y$.  The viscosity supersolution
inequality would give
$\lambda u_\lambda(y)+V(y)\ge0$,
contradicting $u_\lambda(y)<0$ and
$V(y)\le0$.

The absence of an interior local minimum implies
$v(y)\ge\min\{v(x),v(z)\}$ whenever $x<y<z$. Otherwise, $v$ would attain
its minimum on $[x,z]$ at an interior point. Choose $\sigma\in J$ at
which $v$ attains its maximum. Taking $z=\sigma$ shows that $v$ is
nondecreasing to the left of $\sigma$, while taking $x=\sigma$ shows that
$v$ is nonincreasing to the right.

Since $v$ is absolutely continuous and
\[
 |v'(x)|=\sqrt{Q(x)^2-2\lambda u_\lambda(x)}>0
\]
for almost every $x\in J$, monotonicity forces the positive root almost
everywhere on the left of $\sigma$ and the negative root almost everywhere
on the right.  The endpoint choices of $\sigma$ give the two cases in which
only one root is used.
\end{proof}

We can now prove the sharp lower bound.

\begin{proposition}
For the Hamiltonian \eqref{eq:twowell-H} constructed above, there exist
$C_0>0$ and $\lambda_0>0$ such that
\begin{equation}\label{eq:twowell-lower}
 \|u_\lambda-u_0\|_\infty
 \ge C_0\lambda|\log\lambda|
 \qquad\text{for }0<\lambda\le\lambda_0.
\end{equation}
Consequently, Theorem~\ref{thm:finite-main} gives
\[
 \|u_\lambda-u_0\|_\infty
 \asymp\lambda|\log\lambda|.
\]
\end{proposition}

\begin{proof}
Let $\Delta>0$ be given by \eqref{eq:twowell-Delta}.  Since $u_0(1/2)=-\Delta$, choose $\rho>0$ so small that
\[
 u_0(x)\le-\frac{3\Delta}{4}
 \qquad\text{for }x\in\left[\frac12-\rho,\frac12\right].
\]
Since $u_\lambda\to u_0$ uniformly, for all sufficiently small $\lambda$,
\begin{equation}\label{eq:twowell-ulambda-negative}
 u_\lambda(x)\le-\frac{\Delta}{2}
 \qquad\text{for }x\in\left[\frac12-\rho,\frac12\right].
\end{equation}
In particular Lemma~\ref{lem:twowell-branch} applies on this whole interval.
Let $\sigma_\lambda\in[\frac12-\rho,\frac12]$ be the transition point
provided by Lemma~\ref{lem:twowell-branch}.  Thus
$\sigma_\lambda=1/2$ is allowed when the positive branch is used throughout
the interval, and $\sigma_\lambda=1/2-\rho$ is allowed when the negative
branch is used throughout.  Almost everywhere,
\begin{equation}\label{eq:twowell-branch-sigma}
 \begin{cases}
 u_\lambda'+a=\sqrt{Q^2-2\lambda u_\lambda},&1/2-\rho<x<\sigma_\lambda,\\
 u_\lambda'+a=-\sqrt{Q^2-2\lambda u_\lambda},&\sigma_\lambda<x<1/2.
 \end{cases}
\end{equation}
Set
\[
 w_\lambda:=u_\lambda-u_0,
 \qquad
 \ell_\lambda:=\frac12-\sigma_\lambda.
\]

We first estimate how far $\sigma_\lambda$ can lie from $1/2$.  On
$(\sigma_\lambda,1/2)$, by
\eqref{eq:twowell-u0-branch} and \eqref{eq:twowell-branch-sigma},
\[
 w_\lambda'
 =-\sqrt{Q^2-2\lambda u_\lambda}-Q
 \le-Q
\]
almost everywhere.  Hence
\begin{equation}\label{eq:twowell-switch-integral}
 w_\lambda(\sigma_\lambda)-w_\lambda(1/2)
 \ge\int_{\sigma_\lambda}^{1/2}Q(x)\,\dd x.
\end{equation}
The nondegeneracy $V''(1/2)<0$ implies
\[
 Q(1/2-y)=\sqrt{-2V(1/2-y)}
 =\sqrt{-V''(1/2)}\,y+o(y)
 \qquad\text{as }y\to0^+.
\]
After decreasing $\rho$ once more, there are constants $k_1,k_2>0$ such
that
\begin{equation}\label{eq:twowell-Q-linear}
 k_1y\le Q(1/2-y)\le k_2y
 \qquad\text{for }y\in[0,\rho].
\end{equation}
Therefore
\begin{equation}\label{eq:twowell-ell-lower}
 \int_{\sigma_\lambda}^{1/2}Q(x)\,\dd x
 =\int_0^{\ell_\lambda}Q(1/2-y)\,\dd y
 \ge\frac{k_1}{2}\ell_\lambda^2.
\end{equation}
On the other hand, Theorem~\ref{thm:finite-main} applies to this Hamiltonian,
since its Aubry set consists of the two hyperbolic equilibria found in
Lemma~\ref{lem:twowell-geometry}.  We have
\[
 -C\lambda\le w_\lambda(x)
 \le C\lambda|\log\lambda|
 \qquad\text{for }x\in\T.
\]
Combining this estimate with \eqref{eq:twowell-switch-integral} and
\eqref{eq:twowell-ell-lower}, and using $|\log\lambda|\ge1$, we obtain
\begin{equation}\label{eq:twowell-ell}
 \ell_\lambda
 \le C\sqrt{\lambda|\log\lambda|}.
\end{equation}
For small $\lambda$, \eqref{eq:twowell-ell} gives
$\sigma_\lambda>1/2-\rho$.  Hence the negative branch cannot be used
throughout the interval. If it is present, it is confined to an interval of
length $O(\sqrt{\lambda|\log\lambda|})$ adjacent to $1/2$.  The endpoint
case $\sigma_\lambda=1/2$, corresponding to the positive branch throughout,
is still allowed.

We now integrate the error created by that positive branch.  From \eqref{eq:twowell-ulambda-negative},
\[
 \sqrt{Q(x)^2-2\lambda u_\lambda(x)}
 \ge\sqrt{Q(x)^2+\Delta\lambda}
 \qquad\text{for }x\in\left[\frac12-\rho,\frac12\right].
\]
On $(1/2-\rho,\sigma_\lambda)$, using again
\eqref{eq:twowell-u0-branch},
\[
 w_\lambda'=\sqrt{Q^2-2\lambda u_\lambda}-Q.
\]
We have
\begin{equation}\label{eq:twowell-positive-integral}
 \begin{aligned}
 w_\lambda(\sigma_\lambda)
 &\ge w_\lambda(1/2-\rho)
 +\int_{1/2-\rho}^{\sigma_\lambda}
 \left(\sqrt{Q(x)^2+\Delta\lambda}-Q(x)\right)\,\dd x\\
 &\ge -C\lambda
 +\int_{\ell_\lambda}^{\rho}
 \left(\sqrt{Q(1/2-y)^2+\Delta\lambda}-Q(1/2-y)\right)\,\dd y,
 \end{aligned}
\end{equation}
where in the last line we used the lower bound in
Theorem~\ref{thm:finite-main} at the fixed point $1/2-\rho$.

For $s\ge0$,
\[
 \frac{\dd}{\dd s}\left(\sqrt{s^2+\Delta\lambda}-s\right)
 =\frac{s}{\sqrt{s^2+\Delta\lambda}}-1<0.
\]
By the upper bound in \eqref{eq:twowell-Q-linear},
\begin{equation}\label{eq:twowell-integrand-compare}
 \sqrt{Q(1/2-y)^2+\Delta\lambda}-Q(1/2-y)
 \ge
 \sqrt{k_2^2y^2+\Delta\lambda}-k_2y.
\end{equation}
Set
\[
 r_\lambda
 :=\max\left\{\ell_\lambda,
 \frac{\sqrt{\Delta\lambda}}{k_2}\right\}.
\]
By \eqref{eq:twowell-ell},
\begin{equation}\label{eq:twowell-r}
 r_\lambda\le C\sqrt{\lambda|\log\lambda|}
\end{equation}
for all sufficiently small $\lambda$.  If $y\ge r_\lambda$, then
$k_2y\ge\sqrt{\Delta\lambda}$, and therefore
\[
 \sqrt{k_2^2y^2+\Delta\lambda}\le\sqrt2\,k_2y.
\]
Rationalizing the square root now gives
\begin{equation}\label{eq:twowell-one-over-y}
 \sqrt{k_2^2y^2+\Delta\lambda}-k_2y
 =\frac{\Delta\lambda}
 {\sqrt{k_2^2y^2+\Delta\lambda}+k_2y} \ge
 \frac{\Delta}{(\sqrt2+1)k_2}\frac{\lambda}{y}.
\end{equation}
For small $\lambda$, \eqref{eq:twowell-r} also gives $r_\lambda<\rho$.
Using \eqref{eq:twowell-integrand-compare} and
\eqref{eq:twowell-one-over-y} in \eqref{eq:twowell-positive-integral}, we
obtain
\begin{equation}\label{eq:twowell-log-integral}
 w_\lambda(\sigma_\lambda)
 \ge
 -C\lambda
 +C_1\lambda\log\frac{\rho}{r_\lambda}
\end{equation}
for a constant $C_1>0$ independent of $\lambda$.

Finally, \eqref{eq:twowell-r} gives
\[
 \log\frac{\rho}{r_\lambda}
 \ge \log\frac{\rho}{C\sqrt{\lambda|\log\lambda|}}
 =\frac12|\log\lambda|-\frac12\log|\log\lambda|-C,
\]
for a constant $C$ independent of $\lambda$.
Since $\log|\log\lambda|=o(|\log\lambda|)$ as $\lambda\to0^+$, after
reducing $\lambda_0$ we have
\[
 \log\frac{\rho}{r_\lambda}\ge\frac14|\log\lambda|.
\]
The $O(\lambda)$ term in \eqref{eq:twowell-log-integral} is then absorbed by
the logarithmic term, and we conclude that
\[
 w_\lambda(\sigma_\lambda)
 \ge C_0\lambda|\log\lambda|
\]
for some $C_0>0$.  This proves \eqref{eq:twowell-lower}.  The upper
bound follows from Theorem~\ref{thm:finite-main}.
\end{proof}

\section*{Acknowledgements}

The author is grateful to Professor Hiroyoshi Mitake for helpful comments and suggestions, to Professor Hung V. Tran for sharing the construction
used in Subsection~\ref{subsec:arbitrarily-slow}, and to Professor Jun Yan and
Dr.~Kai Zhao for helpful discussions related to this problem. The author is
supported by the JSPS grant: KAKENHI \# 26KF0103 and the National Natural Science
Foundation of China (Grant No.~12571197).

\section*{Declarations}

\noindent {\bf Conflict of interest statement:} The author states that there is no conflict of interest.

\medskip

\noindent {\bf Data availability statement:} Data sharing is not applicable to this article as no datasets were generated or analyzed during the current study.

\medskip

\noindent {\bf AI declaration:} The author used OpenAI's ChatGPT as an auxiliary tool for checking calculations, improving the presentation of some arguments, and polishing the English. All mathematical arguments were independently verified by the author, who takes full responsibility for the content.

\end{document}